\documentclass[aop]{imsart}

\RequirePackage[OT1]{fontenc}
\RequirePackage{amsthm,amsmath}
\RequirePackage[numbers]{natbib}
\RequirePackage[colorlinks,citecolor=blue,urlcolor=blue]{hyperref}
\RequirePackage{natbib}

\usepackage{eurosym}
\usepackage{amssymb}
\usepackage{amsfonts}
\usepackage{enumerate}
\usepackage[normalem]{ulem}
\usepackage{color}
\usepackage[dvipsnames]{xcolor}

\usepackage[margin=1in]{geometry}

\newtheorem{theorem}{Theorem}[section]

\newtheorem{lemma}[theorem]{Lemma}
\newtheorem{proposition}[theorem]{Proposition}

\theoremstyle{remark}

\theoremstyle{remark}
\newtheorem{remark}[theorem]{Remark}
\def\MR#1{\href{http://www.ams.org/mathscinet-getitem?mr=#1}{MR#1}}

\newcommand{\eps}{\varepsilon}

\newcommand{\old}[1]{{}}

\newcommand{{\pe}}  {\partial_e}
\newcommand {{\lodd}} {{\mathcal J}}
\newcommand{{\inrad}} {{\rm inrad}}

\newcommand {{\cent}} {{w_0}}
\newcommand {{\eb}}  {{\bf e}}
\newcommand{{\niceset}}  {{\mathcal X}}
\newcommand {{\nicesettwo}} {\tilde{\niceset}}
\newcommand {{\sprob}}{\mathcal P}
\newcommand {{\smallbit}}{{\mathcal S}}
\newcommand {{\Prob}}{{\mathbb P}}
\newcommand {{\eset}}  {{\emptyset}}
\newcommand {{\Es}}  {{\rm Es}}
\newcommand {{\Cp}} {{\rm cap}}
\newcommand {{\p}}  {{\partial}}
\newcommand {{\hm}}  {{\rm hm}}

\newcommand {{\LE}}  {{\mathbf{LE}}}

\newcommand {{\saws}}  {{\mathcal K}}

\arxiv{arXiv:1608.02987}

\startlocaldefs
\numberwithin{equation}{section}
\theoremstyle{plain}

\endlocaldefs

\begin{document}

\begin{frontmatter}
\title{Maximum of the Ginzburg-Landau fields}
\runtitle{Maximum of the Ginzburg-Landau fields}

\begin{aug}
\author{\fnms{David} \snm{Belius}\ead[label=e1]{ david.belius@unibas.ch}},

\and
\author{\fnms{Wei} \snm{Wu} \ead[label=e3]{w.wu.9@warwick.ac.uk}
}

\runauthor{D. Belius and W. Wu}


\address{David Belius\\
SNF Assistant Professor\\
Department of Mathematics and Computer Science\\ University of Basel\\ Basel, Switzerland\\
\printead{e1}
}

\address{Wei Wu (corresponding author)\\
Assistant Professor\\
Statistics Department\\
University of Warwick\\
Coventry\\
CV4 7AL, UK\\
\printead{e3}
}
\end{aug}

\begin{abstract}
We study two dimensional massless field in a box with potential $V\left(
\nabla \phi \left( \cdot \right) \right) $ and zero boundary condition,
where $V$ is any symmetric and uniformly convex function. Naddaf-Spencer 
\cite{NS} and Miller \cite{M} proved that the rescaled macroscopic averages
of this field converge to a continuum Gaussian free field. In this paper we
prove that the distribution of local marginal $\phi \left( x\right) $, for
any $x$ in the bulk, has a Gaussian tail. We further characterize the
leading order of the maximum and the dimension of high points of this field,
thus generalizing the results of Bolthausen-Deuschel-Giacomin \cite{BDG}\
and Daviaud \cite{Dav}\ for the discrete Gaussian free field.
\end{abstract}

\begin{keyword}[class=MSC]
\kwd{60G50, 60K35}
\end{keyword}
\begin{keyword}
\kwd{\LaTeXe}
\end{keyword}

\end{frontmatter}

\section{ Introduction}

\subsection{Model}

This paper studies the extreme values of certain two-dimensional lattice
gradient Gibbs measures (also known as the Ginzburg-Landau field). Take a
nearest neighbor potential $V\in C^{2}\left( \mathbb{R}\right) $ that
satisfies:%
\begin{gather}
\text{ }V\left( x\right) =V\left( -x\right) ,  \label{a1} \\
\text{ }0<c_{-}\leq V^{\prime \prime }\left( x\right) \leq c_{+}<\infty 
\text{,}  \label{a2}
\end{gather}%
{where }$c_{-},c_{+}${\ are positive constants.}

Let $D_{N}:=\left[ -N,N\right] ^{2}\cap \mathbb{Z}^{2}$ and $\partial D_{N}$
consist of the vertices in $D_{N}$ that are connected to $\mathbb{Z}%
^{2}\setminus D_{N}$ by some edge. The Ginzburg-Landau Gibbs measure on $%
D_{N}$ with zero boundary condition is given by 
\begin{equation}
d\mu _{N}=Z_{N}^{-1}\exp \left[ -\sum_{x\in D_{N}}\sum_{i=1}^{2}V\left(
\nabla _{i}\phi \left( x\right) \right) \right] \prod_{x\in D_{N}\backslash
\partial D_{N}}d\phi \left( x\right) \prod_{x\in \partial D_{N}}\delta
_{0}\left( \phi \left( x\right) \right) ,  \label{GL}
\end{equation}%
where $\nabla _{i}\phi \left( x\right) =\phi \left( x+e_{i}\right) -\phi
\left( x\right) $ {for $e_{1}=(1,0)$ and $e_{2}=(0,1)$}, 
\begin{equation*}
\delta _{0}\left( y\right) =\left\{ 
\begin{array}{cc}
1 & y=0 \\ 
0 & \text{else}%
\end{array}%
\right. ,
\end{equation*}%
and $Z_{N}$ is the normalizing constant such that $\mu _{N}$ is a
probability measure. We denote by $\mathbb{E}^{D_{N},0}$ and Var$^{D_{N},0}$
the expectation and variance with respect to the measure $\mu _{N}$. The
Ginzburg-Landau model is a natural generalization of the discrete Gaussian
free field (DGFF, corresponding to the case $V\left( x\right) =x^{2}/2$). It
is no longer Gaussian in general, but still log-correlated in two
dimension. In fact, one can prove that the limit 
\begin{equation}
\lim_{N\rightarrow \infty }\frac{\text{Var}^{D_{N},0}\phi \left( 0\right) }{%
\log N}=g\text{, \ \ for some }g=g\left( V\right) >0  \label{var}
\end{equation}%
exists; this follows, for example, by following the proof of Theorem \ref%
{thm: tail bound} in this paper. The constant $g=g\left( V\right) $ is known
as the \textbf{effective stiffness} of the random surface model.

\subsection{Results}

Our main result concerns the maximum of the Ginzburg-Landau field $\phi $ in 
$D_{N}$. For potential $V\left( \cdot \right) $ satisfying (\ref{a1}) and (%
\ref{a2}), the well-known Brascamp-Lieb inequalities (Lemma \ref{BL}) imply that with high probability, $\frac{\sup_{v\in D_{N}}\phi
\left( v\right) }{\log N}$ is uniformly bounded above by a constant
depending only on $c_{-}$ (see \cite{DG} where a constant lower bound was
also obtained, and Remark \ref{rem: BL upper bound} below). We prove that
this random variable in fact satisfies a law of large numbers, with more
precise tail bounds given by (\ref{ub}) and (\ref{lb}) below.

\begin{theorem}
\label{main}Let $\phi $ be sampled from the Gibbs measure (\ref{GL}). Assume
the potential $V\left( \cdot \right) $ satisfies (\ref{a1}) and (\ref{a2}).
Then there is a constant{\ $g=g\left( V\right) $}, such that%
\begin{equation}
\frac{\sup_{v\in D_{N}}\phi \left( v\right) }{\log N}\rightarrow 2\sqrt{g}%
\text{ in }L^{2}\text{.}  \label{eq:max}
\end{equation}
\end{theorem}

\begin{remark}
The explicit dependence of $g(V)$ on $V$ is not known. The same constant
also appears in the covariance of the continuum Gaussian Free Field that
emerges as the scaling limit of the measure (\ref{GL}), see \cite{NS},\cite%
{GOS}. One can give a variational characterization of $g(V)$ (see, e.g., 
\cite{GOS},\cite{AW}).
\end{remark}

Theorem \ref{main} is known for the discrete Gaussian Free Field (see \cite%
{BDG}), but not for any other Ginzburg-Landau fields. We will summarize
related results in Section \ref{hist} below. The upper bound of Theorem \ref%
{main} will be proved in Section \ref{sec: ub} and the lower bound in
Section \ref{LB}.

Our next result studies the fractal structure of the sets where the
Ginzburg-Landau field $\phi $ is unusually high. We say that $v\in D_{N}$ is
an $\eta $-high point for the Ginzburg-Landau field if $\phi \left( v\right)
\geq 2\sqrt{g}\eta \log N$. The following theorem generalizes the dimension
of the high points for Gaussian free field, obtained by Daviaud \cite{Dav}.

\begin{theorem}
\label{thm: high points} Denote by $\mathcal{H}_{N}\left( \eta \right)
=\left\{ v\in D_{N}:\phi \left( v\right) \geq 2\sqrt{g}\eta \log N\right\} $
the set of $\eta $-high points. Then for any $\eta \in \left( 0,1\right) $,%
\begin{equation}
\frac{\log \left\vert \mathcal{H}_{N}\left( \eta \right) \right\vert }{\log N%
}\rightarrow 2\left( 1-\eta ^{2}\right)  \label{eq: high points}
\end{equation}%
in probability.
\end{theorem}

{This result is consistent with the conjecture that the level sets of the
Ginzburg-Landau model with zero boundary condition converge to CLE(4), a
collection of conformally invariant random loops (see \cite{SW} for the
definition of the CLE and how to construct a coupling with GFF).} Theorem %
\ref{thm: high points} will be proved in Section \ref{hp}. The main step in
the proofs of the upper bound (\ref{ub}) and the upper bound of (\ref{eq:
high points}) is the following pointwise tail bound for the Ginzburg-Landau
field (\ref{GL}). Here and in the sequel of the paper, for a set $A\subset 
\mathbb{Z}^{2}$ and a point $v\in \mathbb{Z}^{2}$, we use \textup{\textup{%
dist}}$(v,A)$ to denote the (lattice) distance from $v$ to $A$.

\begin{theorem}
\label{thm: tail bound} Let $g$ be the constant as in Theorem \ref{main}.
Given any $C<\infty $, we have for all $v\in D_{N}$, and all $0<u<C\log
\text{dist} \left( v,\partial D_{N}\right)  ,$ 
\begin{equation}
\mathbb{P}^{D_{N},0}\left( \phi \left( v\right) \geq u\right) \leq \exp
\left( -\frac{u^{2}}{2g\log dist\left( v,\partial D_{N}\right) }+o\left(
\log dist\left( v,\partial D_{N}\right) \right) \right) .  \label{inn}
\end{equation}
\end{theorem}

\bigskip The tail bound \eqref{inn} was only known for a class of potentials 
$V\left( \cdot \right) $ that has elliptic contast at most $2$ (i.e., $%
c_{0}\leq V^{\prime \prime }\leq 2c_{0}$, for some $c_{0}>0$) and bounded
third derivative, and $\phi $ is the \textit{infinite volume} limit of the
Gibbs measure (\ref{GL}) (see \cite{CS}). Theorem \ref{thm: tail bound} will
be proved in Section \ref{ptwise}.

\subsection{Historical Survey\label{hist}}

\subsubsection{Ginzburg-Landau fields}

The Gibbs measure (\ref{GL}) was first introduced by Brascamp, Lebowitz and
Lieb, in the name of anharmonic crystals \cite{BLL}. It is believed that the
large scale behaviors of this class of Gibbs measures resemble that of the
Gaussian free field. Rigorous mathematical studies for convex perturbations
of GFF (in particular, the special example called lattice dipole gas) were
initiated by the renormalization group approach of \cite{GK}, and further
developed by \cite{BY}, which confirm its correlation function behaves like
a continuous GFF in the scaling limit. Renormalization group is a powerful
tool to study gradient field models, but it is only applicable in the
perturbative case, i.e., when the potential is given by a small perturbation
of Gaussian, and thus the Hessian of the Hamiltonian is close to the
standard Laplacian. The non-perturbative approach that allows one to study
any convex potential $V$ is based on the Helffer-Sj\"{o}strand formula (\cite%
{HS}, \cite{Hel}) that represents the mean and covariance of such fields in
terms of an elliptic operator (or, probabilisticly, a random walk in dynamic
random environment). We give here an incomplete list of{\ references} that
study the scaling limits of gradient field models. The classification of the
gradient Gibbs states on $\mathbb{Z}^{d}$ were proved by Funaki and Spohn 
\cite{FS}. Deuschel, Giacomin and Ioffe \cite{DGI} studie{d} the large
deviation principle of the macroscopic surface profile in a bounded domain,
where they also introduce the random walk representation of the Helffer-Sj%
\"{o}strand formula. The central limit theorem for linear functionals of the
gradient fields was first established by Naddaf and Spencer \cite{NS} for
the infinite volume gradient Gibbs states with zero tilt (the corresponding
dynamical CLT was proved in \cite{GOS}), and later by Miller \cite{M} for
the gradient fields in bounded domains. It is also proved in \cite{M2} that
the level set for such gradient fields in a bounded domain (with certain
Dirichlet boundary condition) converges to the chordal SLE(4), an example of
the conformally-invariant random curve in the plane known as the
Schramm-Loewner Evolution (for a survey on SLE, see e.g. \cite{Law}).

Nonlinear functionals of the gradient fields are much less known. With
additional bounded ellipticity assumption on $V$, it is proved by Conlon and
Spencer \cite{CS} that for the infinite gradient-Gibbs states with zero
slope, there exists $C<\infty $ such that 
\begin{equation*}
 \left\vert \log \mathbb{E}\left[ e^{t\left( \phi \left( 0\right) -\phi \left(
x\right) \right) }\right] -\frac{t^{2}}{2}\text{Var}\left[ \phi \left(
0\right) -\phi \left( x\right) \right] \right\vert \leq Ct^{3}\left\Vert
V^{\prime \prime \prime }\right\Vert _{\infty }.
\end{equation*}%
Their argument is based on the Helffer-Sj\"{o}strand formula and operator
theory on weighted Hilbert space. This phenomenon is remarkable because it
indicates the pointwise distribution of $\phi \left( 0\right) -\phi \left(
x\right) $ is nearly Gaussian, and one has to go to the large deviation
regime (corresponding to $t=O\left( \log \left\vert x\right\vert \right) $)
to see non-Gaussian tails. In this paper we remove the bounded ellipticiy
assumption, and rely our proof on a different stretegy.

\subsubsection{Extrema of log-correlated random fields}

Although the macroscopic behavior of linear functionals of the gradient
fields are now well understood, finer properties of the field, such as the
behavior of its maximum, remain to be clarified. Question abouts the maximum
fit into the wider context of the study of extrema of log-correlated random
fields.

Multiscale analysis is the key to study the extrema of such random fields. The conceptually simplest cases, which already exhibit the most crucial phenomena
underlying the behavior the extrema, are tree models such as Branching Brownian Motion and
Branching Random Walk. In his seminal work Bramson introduced a truncated second
moment method to study Branching Brownian Motion \cite{BramsonMaxDisplacementofBBM,Bramson1ConvergenceofSolutionsOfKolmogorovEqn}.
This method has been much refined to obtain detailed results to the level of
the convergence of the extremal process in Bramson's setting and for
Branching Random Walk \cite{ABKExtremalProcOFBBM,AidekonBerestickyBruneyShiBBMSeenFromItsTip,AidekonConvinLawofMinofBRW,BramsonDingZeitouni-ConvergenceinLawOfTheMaxOfNonLAtticeBRW,MadauleConvergenceInLawOftheBRWSeenFromItsTip}.

Beyond such processes the most investigated case is the Gaussian Free Field.
The discrete Gaussian Free Field is the special case {$V(x)=\frac{1}{2}x^{2}$%
} in the present set-up. Bolthausen, Deuschel and Giacomin \cite{BDG} first
showed the equivalent of our main result, which was later improved \cite{BZ,
BDZ2} to 
\begin{equation*}
\sup_{x\in D_{N}}\phi \left( x\right) =2\sqrt{g_{0}}\log N-\frac{3}{4}\sqrt{%
g_{0}}\log \log N+O\left( 1\right) ,\text{ as }N\rightarrow \infty ,
\end{equation*}%
where $g_{0}=2/\pi $. Furthermore it has been proved that the $O(1)$ term
converges in law and the geometric properties of the near extrema has been
studied, including the convergence of the extremal process \cite%
{DZ,BL2,BiskupLouidorConformalSymmetries,BiskupLouidorFullExtProcClusterLawEtc}%
. The equivalent of our Theorem \ref{thm: high points} for the discrete
Gaussian Free Field was proved in \cite{Dav}. Some results have been
generalized to a wider class of log-correlated Gaussian fields \cite{DRZ}.

The article \cite{BeliusKister2DCT} studied the extrema of a log-correlated
field that is neither Gaussian nor endowed with an exact tree structure. It
constructed what can be interpreted as a sequence of regularizations of the
field and from these obtained a collection of approximate branching random
walks indexed by the points of the field, to which Bramson's method can be
applied (regularization also plays an important role in problems connected
to the continuum Gaussian Free Field \cite%
{SheffieldGFFforMath,robert2008hydrodynamic,DuplantierSheffieldLQGandKPZ,HuMillerPeresThickPointsofGFF}%
). 

\cite{KistlerDerridasRandomEnergyModelsBeyondSpinGlasses} adapted this
approach to the Gaussian Free Field, with the
regularizations given by harmonic averages on concentric boxes (``local
projections"). It also describes a ``K-level coarse-graining" which is 
a particularly streamlined version of the multiscale argument that provides
leading order estimates for the maximum from minimal technical inputs. Subsequently versions of it has been used to study the
extrema of many cases of non-Gaussian log-correlated random fields \cite%
{ArguinBeliusHarper-RandomZeta,ArgBelBou15,PaqZei16,ChaMadNaj16,LamPaq16}.

\subsection{Proof strategy}


To prove the tail bound \eqref{inn}, the estimates \eqref{eq:max} for the
maximum and \eqref{eq: high points} for the high-points we adapt the aforementioned
local projections and $K-$
level coarse-graining of \cite{KistlerDerridasRandomEnergyModelsBeyondSpinGlasses}.
Namely, we consider the harmonic average over circles of the field around each point, as a
process indexed by the side-length of the box, and use the first moment
method to obtain an upper bound for the maximum and a truncated second
moment argument involving the average process to get a lower bound for the
maximum. These average processes are expected to evolve similarly to
branching random walks as one varies the side-length of the box at dyadic
scales. For the Gaussian Free Field, Gaussian orthogonal decomposition
implies the increments of such harmonic averages are independent, making the
random walk approximation fairly straightforward. This fails for the general
gradient field models studied in this paper. In fact, one of the main
contributions of this paper is to prove the asymptotic decoupling of these
increments (Theorem \ref{ind}). We apply the useful tool from \cite{M}, that
gives an approximate harmonic coupling of the Ginzburg-Landau field on a
bounded domain with different boundary conditions. Inspired by the $K$-level
coarse-graining of \cite{KistlerDerridasRandomEnergyModelsBeyondSpinGlasses}
we exploit that for the level of accuracy we seek in the present paper, it
is enough to consider the behaviour of the approximate random walks over a relatively
small number of large increments, corresponding to a small number of scales
(only finitely many in the case of the Gaussian Free Field; for technical
reasons we use a slowly growing number of increments). The approximate
harmonic coupling allows us to show that each increment of the harmonic
average, conditioned on the Ginzburg-Landau field outside, is distributed
not far from a Gaussian, after discarding a thin layers between each scale.
This gives the pointwise tail bound for the Ginzburg-Landau field, and thus
also the upper bound in Theorem \ref{main}. A similar argument via the truncated second moment method gives the
two-point tail bounds needed to obtain the lower bound in Theorem \ref{main}.

\subsection{Open question}

We finish the introduction with a corresponding open question for dimer
models. A (uniform) dimer model on $\mathbb{Z}^{2}$ can be{\ thought} of as
an integer valued random surface $h\left( v\right) ,v\in \mathbb{Z}^{2}$. It
is an integrable model with determinantal structure. It is shown in \cite{K1}
and \cite{K2} that the height fluctuation $h\left( 0\right) -h\left(
v\right) $ has logarithmic variance, and moreover the rescaled height
function converges weakly to GFF. A main conjecture in this field is that
the level sets of the height function converges to CLE(4). Still, it would
be very interesting to prove the maximum of the dimer height function
satisfies Theorem \ref{main}. The method in the present paper does not apply
directly because the harmonic coupling (see Section \ref{sec:hc}) have not
yet been established for the dimer model.

\section{Tools}

\subsection{Brascamp-Lieb inequality}

One can bound the variances and exponential moments with respect to the
Ginzburg-Landau measure by those with respect to the Gaussian measure, using
the following Brascamp-Lieb inequality. Let $\phi $ be sampled from the
Gibbs measure (\ref{GL}), with a nearest-neighbor potential $V\in
C^{2}\left( \mathbb{R}\right) $ that satisfies $\inf_{x\in \mathbb{R}%
}V^{\prime \prime }\left( x\right) \geq c_{-}>0$. Given $f\in \mathbb{R}%
^{D_{N}}$, we define%
\begin{equation*}
\left\langle \phi ,f\right\rangle :=\sum_{x\in D_{N}}\phi \left( x\right)
f\left( x\right) .
\end{equation*}

\begin{lemma}[Brascamp-Lieb inequalities \protect\cite{BL}] \label{BL}
Let $\mathbb{E}_{DGFF}^{D_{N},0}$ and Var$_{DGFF}^{D_{N},0}$ denote the
expectation and variance with respect to the discrete GFF\ measure (that is, (\ref%
{GL}) with $V\left( x\right) =x^{2}/2$). Then for any $f\in \mathbb{R}%
^{D_{N}}$, 
\begin{align}
\text{Var}^{D_{N},0}\left\langle \phi ,f\right\rangle \quad & \leq & &
c_{-}^{-1}Var_{DGFF}^{D_{N},0}\left\langle \phi ,f\right\rangle ,
\label{eq: BL var bound} \\
\mathbb{E}^{D_{N},0}\left( \left\langle \phi ,f\right\rangle -\mathbb{E}%
^{D_{N},0}\left\langle \phi ,f\right\rangle \right) ^{2k}\quad & \leq & &
c_{-}^{-k}\mathbb{E}_{DGFF}^{D_{N},0}\left( \left\langle \phi
,f\right\rangle -\mathbb{E}_{DGFF}^{D_{N},0}\left\langle \phi
,f\right\rangle \right) ^{2k},\text{ for }k\in \mathbb{N}
\label{eq: BL even moment bound} \\
\mathbb{E}^{D_{N},0}\left[ \exp \left( \left\langle \phi ,f\right\rangle -%
\mathbb{E}^{D_{N},0}\left\langle \phi ,f\right\rangle \right) \right] \quad
& \leq & & \exp \left( \frac{1}{2}c_{-}^{-1}Var_{DGFF}^{D_{N},0}\left\langle
\phi ,f\right\rangle \right) .  \label{eq: BL exp mom bound}
\end{align}
\end{lemma}

The Brascamp-Lieb inequalities can be used to show the following a-priori
tail bound for $\phi .$

\begin{lemma}
\label{lem: BL tail bound} There is a positive constant $c_{BL}$ such that 
\begin{equation}
\mathbb{P}^{D_{N},0}\left( \phi (v)\geq u\right) \leq e^{-c_{BL}\frac{u^{2}}{%
dist(v,\partial D_{N})}},\text{ for }v\in D_{N}.  \label{eq: BL tail bound}
\end{equation}
\end{lemma}

\begin{proof}
By Chebyshev's inequality,%
\begin{equation*}
\mathbb{P}^{D_{N},0}\left( \phi (v)\geq u\right) \leq e^{-tu}\mathbb{E}%
^{D_{N},0}\exp \left( t\phi \left( v\right) \right) .
\end{equation*}%
Applying the Brascamp-Lieb inequality with $f=\delta _{v}$, and using the
fact that (see the Green's function asymptotics in \cite{LL}) 
\begin{equation*}
\text{Var}_{DGFF}^{D_{N},0}\phi \left( v\right) = G_{D_N}(v,v) = \sqrt{2/\pi }\log
dist\left( v,\partial D_{N}\right) +O\left( 1\right) ,
\end{equation*}%
we have 
\begin{equation*}
\mathbb{P}^{D_{N},0}\left( \phi (v)\geq u\right) \leq \exp \left( -tu+\frac{%
t^{2}}{2}c_{1}\log \text{dist}\left( v,\partial D_{N}\right) \right) .
\end{equation*}%
Optimizing over $t$ then yields the result.
\end{proof}

\begin{remark}
\label{rem: BL upper bound} By a union bound over the $(2N+1)^{2}$ points of 
$D_{N}$ and take $u\gg \sqrt{1/c_{BL}}\log N$, so that the right-hand side
of \eqref{eq: BL tail bound} is $\ll N^{-2}$, one obtains an upper bound of $%
\sqrt{1/c_{BL}}\log N$ for the maximum of $\phi (v)$. This is an upper bound
of the right order, but the constant in front of $\log N$ is larger than the ''true" one $2\sqrt{g}$.
\end{remark}

\subsection{Approximate harmonic coupling\label{sec:hc}}

By definition, the Ginzburg-Landau measures satisfy the domain Markov
property: conditioned on the values on the boundary of a domain, the field
inside the domain is again a gradient field with boundary condition given by
the conditioned values. For the discrete GFF, there is in addition a nice
orthogonal decomposition. More precisely, the conditioned field inside the
domain is the discrete harmonic extension of the boundary value to the whole
domain plus an \textit{independent }copy of a \textit{zero boundary}
discrete GFF.

While this exact decomposition does not carry over to general
Ginzburg-Landau measures, the next result due to Jason Miller, see \cite{M},
provides an approximate version. For $D\subset \mathbb{Z}^{2}$, define the
Ginzburg-Landau measure on $D$ with Dirichlet boundary condition $f$ by 
\begin{equation}
d\mu _{D}^{f}=Z_{D}^{-1}\exp \left[ -\sum_{x\in D}\sum_{i=1}^{2}V\left(
\nabla _{i}\phi \left( x\right) \right) \right] \prod_{x\in D\backslash
\partial D}d\phi \left( x\right) \prod_{x\in \partial D}\delta _{0}\left(
\phi \left( x\right) -f\left( x\right) \right) .  \label{GLD}
\end{equation}

\begin{theorem}[Theorem 1.2 in \protect\cite{M}]
\label{decouple} Let $D\subset \mathbb{Z}^{2}$ be a simply connected domain
of diameter $R$, and denote $D^{r}=\left\{ v\in D:\text{dist}(v,\partial
D)>r\right\} $. Let $\Lambda $ be such that $f:\partial D\rightarrow \mathbb{%
R}$ satisfies $\max_{x\in \partial D}\left\vert f\left( x\right) \right\vert
\leq \Lambda \left\vert \log R\right\vert ^{\Lambda }$. Let $\phi $ be
sampled from the measure (\ref{GLD}) with zero boundary condition, and $\phi
^{f}$ be sampled from the measure (\ref{GLD}) with boundary condition $f$.
Then there exist constants $c,\gamma ,\delta \in \left( 0,1\right) $, that
only depend on $V$, so that if $r>cR^{\gamma }$ then the following holds.
There exists a coupling $\left( \phi ,\phi ^{f}\right) $, such that if $\hat{%
\phi}:D^{r}\rightarrow \mathbb{R}$ is discrete harmonic with $\hat{\phi}%
|_{\partial D^{r}}=\left( \phi ^{f}-\phi \right) |_{\partial D^{r}}$, then 
\begin{equation*}
\mathbb{P}\left( \phi ^{f}=\phi +\hat{\phi}\text{ in }D^{r}\right) \geq
1-c\left( \Lambda \right) R^{-\delta }.
\end{equation*}
\end{theorem}

\bigskip An immediate application of Theorem \ref{decouple} shows that the
mean of a Ginzburg-Landau field at one point in the bulk is approximately
(discrete) harmonic.

\begin{theorem}[Theorem 1.3 in \protect\cite{M}]
\label{mean}Suppose the same conditions in Theorem \ref{decouple} holds. Let 
$\phi ^{f},c,\gamma ,\delta ,D^{r}$ be defined as in Theorem \ref{decouple}.
For all $r>cR^{\gamma }$, and discrete harmonic function $\hat{\phi}%
:D^{r}\rightarrow \mathbb{R}$ with $\hat{\phi}|_{\partial D^{r}}=\mathbb{E}%
\phi ^{f}|_{\partial D^{r}}$, then 
\begin{equation*}
\max_{v\in D^{r}}\left\vert \mathbb{E}\phi ^{f}\left( v\right) -\hat{\phi}%
\left( v\right) \right\vert \leq c^{\prime }\left( \Lambda \right)
R^{-\delta }.
\end{equation*}
\end{theorem}

Theorem \ref{decouple} allows to compare a Ginzburg-Landau field with
non-zero boundary condition with one that has zero boundary condition. Since
Theorem \ref{decouple} requires that the function $f$ is not too large, we
introduce the \textquotedblleft good\textquotedblright\ event 
\begin{equation*}
\mathcal{G}\left( c\right) =\left\{ \phi :\max_{v\in D}\left\vert \phi
\left( v\right) \right\vert <c\left( \log R\right) ^{2}\right\} ,
\end{equation*}%
which is typical since even using only Brascamp-Lieb one has that $%
\max_{v\in D}|\phi (v)|\leq O\left( \log R\right) $ with high probability.
Indeed, we have

\begin{lemma}
\label{bad}There is some $c_{1}=c_{1}\left( c\right) >0$, such that $\mathbb{%
P}^{D,0}\left( \mathcal{G}\left( c\right) ^{c}\right) \leq \exp \left(
-c_{1}\left( \log R\right) ^{3}\right) $.
\end{lemma}

\begin{proof}
By the union bound,%
\begin{equation*}
\mathbb{P}^{D,0}\left( \mathcal{G}^{c}\right) \leq \sum_{v\in D}\mathbb{P}%
^{D,0}\left( \left\vert \phi \left( v\right) \right\vert >c\left( \log
R\right) ^{2}\right) .
\end{equation*}%
We apply Lemma \ref{lem: BL tail bound} to obtain 
\begin{equation*}
\mathbb{P}^{D,0}\left( \left\vert \phi \left( v\right) \right\vert >c\left(
\log R\right) ^{2}\right) \leq \exp \left( -\left( 4C\right) ^{-1}\left(
\log R\right) ^{3}+O\left( \log R\right) ^{2}\right) ,
\end{equation*}%
for some $C<\infty $, and summing over $v\in D$ then finishes the proof.
\end{proof}

We will use repeatedly the following consequence of Theorem \ref{decouple}%
. It applies to functions $\rho $ such that the integral of $\rho $ against
a harmonic function is always zero.

\begin{lemma}
\label{jasonerr}There exists constants $\delta ,\gamma >0$ such that for any
simply connected $D\subset \mathbb{Z}^{2}$ of diameter $R$, any $r>R^{\gamma
}$ and any $\rho :D\rightarrow \mathbb{R}$ supported on $D^{r}$ that
satisfies $\sum_{x\in D^{r}}\rho \left( x\right) f\left( x\right) =0$ for
all functions $f$ harmonic in $D^{r}$, and $\frac{1}{R}\sum_{y\in
D}\left\vert \rho \left( y\right) \right\vert <\infty $, we have for $R$
large enough, 
\begin{eqnarray*}
&&\left\vert \mathbb{E}^{D,f}\left[ \exp \left( R^{-1}\sum_{x\in D}\rho
\left( x\right) \phi ^{f}\left( x\right) \right) 1_{\mathcal{G}}\right] -%
\mathbb{E}^{D,0}\left[ \exp \left( R^{-1}\sum_{x\in D}\rho \left( x\right)
\phi \left( x\right) \right) 1_{\mathcal{G}}\right] \right\vert \\
&\leq &2\exp \left( c\text{Var}_{DGFF}^{D,0}\left( R^{-1}\sum_{x\in D}\rho
\left( x\right) \phi \left( x\right) \right) \right) R^{-\delta }\text{,}
\end{eqnarray*}%
for some $c<\infty .$
\end{lemma}

\begin{remark}
This lemma is useful if Var$_{DGFF}^{D,0}\left( R^{-1}\sum_{x\in D}\rho
\left( x\right) \phi \left( x\right) \right) \ll \delta \log R$.
\end{remark}

\begin{proof}
Applying Theorem \ref{decouple}, there is an event $\mathcal{C}$ with $%
\mathbb{P}\left( \mathcal{C}^{c}\right) \leq R^{-\delta _{0}}$, where $%
\delta _{0}$ is the constant $\delta $ in Theorem \ref{decouple}, such that
on $\mathcal{C}$ we have $\phi ^{f}-\phi =\hat{\phi}$ in $D^{r}$. Therefore
on $\mathcal{C}$%
\begin{eqnarray*}
\sum_{x\in D}\rho \left( x\right) \phi ^{f}\left( x\right) &=&\sum_{x\in
D^{r}}\rho \left( x\right) \phi ^{f}\left( x\right) =\sum_{x\in D^{r}}\rho
\left( x\right) \phi \left( x\right) +\sum_{x\in D^{r}}\rho \left( x\right) 
\hat{\phi}\left( x\right) \\
&=&\sum_{x\in D^{r}}\rho \left( x\right) \phi \left( x\right) =\sum_{x\in
D}\rho \left( x\right) \phi \left( x\right) ,
\end{eqnarray*}%
where the first and the last equality follows from the fact that $\rho $ is
supported in $D^{r}$. On $\mathcal{C}^{c}$ we apply Holder inequality to
obtain%
\begin{eqnarray}
&&\mathbb{E}^{D,f}\left[ \exp \left( R^{-1}\sum_{x\in D}\rho \left( x\right)
\phi ^{f}\left( x\right) \right) 1_{\mathcal{G}\cap \mathcal{C}^{c}}\right] 
\notag \\
&\leq &\mathbb{P}\left( \mathcal{C}^{c}\right) ^{1/2}\mathbb{E}^{D,f}\left[
\exp \left( 2R^{-1}\sum_{x\in D}\rho \left( x\right) \phi ^{f}\left(
x\right) \right) \right] ^{1/2}  \notag \\
&\leq &R^{-\delta _{0}/2}\mathbb{E}^{D,f}\left[ \exp \left(
2R^{-1}\sum_{x\in D}\rho \left( x\right) \phi ^{f}\left( x\right) -\mathbb{E}%
^{D,f}\left[ 2R^{-1}\sum_{x\in D}\rho \left( x\right) \phi ^{f}\left(
x\right) \right] \right) \right] ^{1/2}  \notag \\
&&\times \exp \left( \mathbb{E}^{D,f}\left[ R^{-1}\sum_{x\in D}\rho \left(
x\right) \phi ^{f}\left( x\right) \right] \right) .  \label{err}
\end{eqnarray}%
By the Brascamp-Lieb inequality (\ref{eq: BL exp mom bound}), there exist
some $c\,<\infty $, such that 
\begin{eqnarray}
&&\mathbb{E}^{D,f}\left[ \exp \left( 2R^{-1}\sum_{x\in D}\rho \left(
x\right) \phi ^{f}\left( x\right) -\mathbb{E}^{D,f}\left[ 2R^{-1}\sum_{x\in
D}\rho \left( x\right) \phi ^{f}\left( x\right) \right] \right) \right] 
\notag \\
&\leq &\exp \left( c\text{Var}_{DGFF}^{D,f}\left( R^{-1}\sum_{x\in D}\rho
\left( x\right) \phi ^{f}\left( x\right) \right) \right) .  \label{errbl}
\end{eqnarray}%
On the other hand, applying Theorem \ref{mean} yields%
\begin{eqnarray}
\left\vert \mathbb{E}^{D,f}\left[ R^{-1}\sum_{x\in D}\rho \left( x\right)
\phi ^{f}\left( x\right) \right] \right\vert &=&\left\vert \mathbb{E}^{D,f}%
\left[ R^{-1}\sum_{x\in D}\rho \left( x\right) \phi ^{f}\left( x\right) %
\right] -R^{-1}\sum_{x\in D}\rho \left( x\right) \hat{\phi}\left( x\right)
\right\vert  \notag \\
&\leq &\left\Vert \mathbb{E}^{D,f}\phi ^{f}-\hat{\phi}\right\Vert
_{L^{\infty }\left( D^{r}\right) }\frac{1}{R}\sum_{x\in D}\left\vert \rho
\left( x\right) \right\vert \leq CR^{-\delta _{0}}.  \label{errmean}
\end{eqnarray}%
Combining (\ref{err}), (\ref{errbl}) and (\ref{errmean}), we have for $R$
large enough, 
\begin{equation*}
\mathbb{E}^{D,f}\left[ \exp \left( R^{-1}\sum_{x\in D}\rho \left( x\right)
\phi ^{f}\left( x\right) \right) 1_{\mathcal{G}\cap \mathcal{C}^{c}}\right]
\leq C\exp \left( c\text{Var}_{DGFF}\left( R^{-1}\sum_{x\in D}\rho \left(
x\right) \phi ^{f}\left( x\right) \right) \right) R^{-\delta _{0}/2}.
\end{equation*}%
And similarly,%
\begin{equation*}
\mathbb{E}^{D,0}\left[ \exp \left( R^{-1}\sum_{x\in D}\rho \left( x\right)
\phi \left( x\right) \right) 1_{\mathcal{G}\cap \mathcal{C}^{c}}\right] \leq
C\exp \left( c\text{Var}_{DGFF}\left( R^{-1}\sum_{x\in D}\rho \left(
x\right) \phi \left( x\right) \right) \right) R^{-\delta _{0}/2}.
\end{equation*}%
Since the variance of linear functionals of Gaussian free field does not
depend on boundary conditions, we finish the proof.
\end{proof}

\subsection{Central limit theorem}

We now state the central limit theorem for macroscopic averages of $\phi $,
proved in \cite{M} as a consequence of Theorem A\ in \cite{NS} and Theorem %
\ref{decouple} stated above.

Let $D\subset \mathbb{R}^{2}$ be a smooth simply connected domain. Before
stating the central limit theorem, we give the definition of the (continuum) 
$a$-Gaussian Free Field ($a$-GFF) $h$ in $D$ with zero boundary condition,
where $a$ is a $2\times 2$ positive definite matrix. The $a$-GFF in $D$ is the standard
Gaussian in $H_{0}^{1}\left( D\right) $, such that for any $f\in
H_{0}^{1}\left( D\right) $, $\int_{D}\nabla h \cdot \nabla f$ is a Gaussian random
variable with mean $0$ and variance $\int_{D}\nabla f\cdot a\nabla f$ $.$

\begin{theorem}
\label{thm: miller CLT} Let $D\subset \mathbb{R}^{2}$ be a smooth simply
connected domain, $D_{N}=D\cap \frac{1}{N}\mathbb{Z}^{2}$ and $\phi $ be
sampled from the Ginzburg-Landau measure on $D_{N}$ with zero boundary
condition. Suppose that the sequence of functions $\rho
_{N}:D_{N}\rightarrow \mathbb{R}$ satisfies 
\begin{equation}
\sum_{x\in D_{N}}\rho _{N}\left( x\right) H\left( x\right) =0\text{, for any
harmonic function }H:D_{N}\rightarrow \mathbb{R}.  \label{harmtest}
\end{equation}%
Also assume there exist some $\rho \in C_{0}^{\infty }\left( D\right) $ and $%
C<\infty $, such that 
\begin{equation}
\int_{D}\rho \left( x\right) H\left( x\right) dx=0\text{, for any harmonic
function }H:D\rightarrow \mathbb{R}, \label{harmtestcont}
\end{equation}%
and $\left\Vert \rho _{N}-\rho \right\Vert _{L^{\infty }\left( D\right)
}\leq C/N$. Then the linear functional%
\begin{equation*}
N^{-1}\sum_{x\in D_{N}}\rho _{N}\left( x\right) \phi \left( x\right)
\end{equation*}%
converges in $L^{2k}$, $k\in \mathbb{N}$, to the random variable 
\begin{equation*}
\int_{D}h\left( x\right) \rho \left( x\right) dx,
\end{equation*}%
where $h$ is the $a$-GFF on $D$ with zero boundary condition, for some $%
a\left( V\right) =\bar{a}\left( V\right) I$ that satisfies $c_{-}\leq \bar{a}%
\leq c_{+}$.
\end{theorem}

\begin{proof}
If $\rho _{N}=\rho $ for all $N\geq 1$, this is a consequence of Theorem 1.1
in \cite{M}, which was proved by combining the CLT for the infinite volume
gradient Gibbs measure (Theorem A\ in \cite{NS}), with the approximate
harmonic coupling (Theorem \ref{decouple}). Below we give a brief
explanation that when $\rho _{N}\rightarrow \rho $ with an algebraic rate,
the CLT for both the infinite volume gradient Gibbs measure (Theorem A\ in 
\cite{NS}) and the finite volume one (Theorem 1.1 in \cite{M}) still hold.

We first sketch the modification for Theorem A\ in \cite{NS}. The theorem
states that for the infinite volume field $\phi ^{0}$, and any $\rho \in
C_{0}^{\infty }\left( D\right) $ such that $\rho =\Delta g$ for some $g\in
H^{1}\left( D\right) $, $N^{-1}\sum_{x\in D_{N}}\rho \left( x\right) \phi^0
\left( x\right) $ converges in $L^{2k}$ to the random variable $%
\int_{D}h\left( x\right) \rho \left( x\right) dx=\int_{D}\nabla h\left(
x\right) \nabla g\left( x\right) dx$, where $h$ is the $a$-GFF in $\mathbb{R}%
^{2}$, for some $a = a(V)$. It suffices to show that 
\begin{equation}
\left\vert \text{Var}\left[ N^{-1}\sum_{x\in D_{N}}\rho _{N}\left( x\right)
\phi ^{0}\left( x\right) \right] -\text{Var}\left[ N^{-1}\sum_{x\in
D_{N}}\rho \left( x\right) \phi ^{0}\left( x\right) \right] \right\vert
\rightarrow 0\text{ \ \ as \ }N\rightarrow \infty .  \label{varapprox}
\end{equation}%
(the deviation in higher moments $k\geq 2$ can then be controlled by the
exponential Brascamp-Lieb inequality (\ref{eq: BL exp mom bound})). Arguing
as in \cite{NS}, let $H^{1}\left( \mathbb{Z}^{2},\mu \right) $
denote the function space consists of $v:\mathbb{Z}^{2}\rightarrow \mathbb{R%
}$ such that 
\begin{equation*}
\mathbb{E}\left[ \frac{1}{N^{2}}\sum_{e\in \frac{1}{N}\mathbb{Z}^{2}}\left(
\nabla v\left( e\right) \right) ^{2}+\frac{1}{N^{2}}\sum_{x,y\in \frac{1}{N}%
\mathbb{Z}^{2}}\left( \frac{\partial }{\partial \phi \left( y\right) }%
v\left( x\right) \right) ^{2}\right] <\infty .
\end{equation*}
Consider the solutions $u_{N},u\in H^{1}\left( \mathbb{Z}%
^{2},\mu \right) $ to the elliptic PDEs in $\mathbb{Z}^{2}$:%
\begin{eqnarray}
-\Delta _{\phi }u_{N}+\nabla ^{\ast } \cdot V^{\prime \prime }\left( \nabla \phi
^{0}\right) \nabla u_{N} &=&\frac{1}{N}\rho _{N},  \label{HSn} \\
-\Delta _{\phi }u+\nabla ^{\ast } \cdot V^{\prime \prime }\left( \nabla \phi
^{0}\right) \nabla u &=&\frac{1}{N}\rho ,  \label{HS}
\end{eqnarray}%
with $\Delta _{\phi }=\sum_{x\in \mathbb{Z}^{2}}\left[ \frac{\partial ^{2}}{%
\partial \phi \left( x\right) ^{2}}-\sum_{y\sim x}V^{\prime }\left( \phi
\left( y\right) -\phi \left( x\right) \right) \frac{\partial }{\partial \phi
\left( x\right) }\right] $. 
For $u,v\in H^{1}\left( \mathbb{Z}^{2},\mu \right) $, define the bilinear
form 
\begin{equation*}
B\left[ u,v\right] =\mathbb{E}\left[ \sum_{x,y\in \frac{1}{N}\mathbb{Z}^{2}}%
\frac{\partial }{\partial \phi \left( y\right) }u\left( x\right) \frac{%
\partial }{\partial \phi \left( y\right) }v\left( x\right) +\sum_{e\in \frac{%
1}{N}\mathbb{Z}^{2}}\nabla u\left( e\right) V^{\prime \prime }\left( \nabla
\phi ^{0}\left( e\right) \right) \nabla v\left( e\right) \right] .
\end{equation*}%
It is observed in \cite{NS} that equations \eqref{HSn} and \eqref{HS} are well-posed, and we may represent the variance in terms of
the bilinear forms as 
\begin{eqnarray*}
\text{Var}\left[ N^{-1}\sum_{x\in D_{N}}\rho _{N}\left( x\right) \phi
^{0}\left( x\right) \right] &=&B\left[ u_{N},\frac{1}{N}\rho _{N}\right] \\
\text{Var}\left[ N^{-1}\sum_{x\in D_{N}}\rho \left( x\right) \phi ^{0}\left(
x\right) \right] &=&B\left[ u,\frac{1}{N}\rho \right] .
\end{eqnarray*}%
Therefore we may reduce the proof of (\ref{varapprox}) to standard energy
comparison for the equations (\ref{HSn}) and (\ref{HS}). Indeed, using%
\begin{eqnarray*}
B\left[ u_{N},\frac{1}{N}\rho _{N}\right] -B\left[ u,\frac{1}{N}\rho \right]
&=&B\left[ u_{N}-u,\frac{1}{N}\rho _{N}-\frac{1}{N}\rho \right] \\
&=&\mathbb{E}\left[ \frac{1}{N}\sum_{x\in D_{N}}\left( \rho _{N}\left(
x\right) -\rho \left( x\right) \right) \left( u_{N}\left( x\right) -u\left(
x\right) \right) \right]
\end{eqnarray*}%
Notice that (\ref{harmtest}) and \eqref{harmtestcont} implies we can write $\frac{1}{N}\left( \rho
_{N}-\rho \right) =\nabla g_{N}$ for some $g_{N}$, and $\left\Vert \rho
_{N}-\rho \right\Vert _{L^{\infty }\left( D\right) }\leq C/N$ suggests that one may take $g_N$ so that $%
\left\Vert g_{N}\right\Vert _{L^{\infty }\left( D_{N}\right) }\leq C/N$.
Also, subtracting (\ref{HSn}) from (\ref{HS}) and testing the new equation
with $u_{N}-u$ yields%
\begin{equation}
\sum_{e\in D_{N}}(\nabla u_{N}\left( e\right) -\nabla u\left( e\right)
)^{2}\leq c_{-}^{-1}\frac{1}{N^{2}}\sum_{x\in D_{N}}\left( \rho _{N}\left(
x\right) -\rho \left( x\right) \right) ^{2}.  \label{LM}
\end{equation}%
Therefore 
\begin{eqnarray*}
&&\left\vert \mathbb{E}\left[ \frac{1}{N}\sum_{x\in D_{N}}\left( \rho
_{N}\left( x\right) -\rho \left( x\right) \right) \left( u_{N}\left(
x\right) -u\left( x\right) \right) \right] \right\vert \\
&=&\left\vert \mathbb{E}\left[ \sum_{e\in D_{N}}g_{N}\left( e\right) (\nabla
u_{N}\left( e\right) -\nabla u\left( e\right) )\right] \right\vert \\
&\leq &\mathbb{E}\left[ \sum_{e\in D_{N}}\left( g_{N}\left( e\right) \right)
^{2}\right] ^{1/2}\left\vert \mathbb{E}\left[ c_{-}^{-1}\frac{1}{N^{2}}%
\sum_{x\in D_{N}}\left( \rho _{N}\left( x\right) -\rho \left( x\right)
\right) ^{2}\right] ^{1/2}\right\vert \\
&\leq &C_{2}N^{-1}\text{, \ for some }C_{2}<\infty \text{.}
\end{eqnarray*}%
Where the last two inequalities follows from (\ref{LM}) and the rate of
convergence for $\rho _{N}$. This concludes the $L^2$ convergence of $%
N^{-1}\sum_{x\in D_{N}}\rho _{N}\left( x\right) \phi ^{0}\left( x\right) $,
to $\int_D h(x) \rho(x) \,dx$. 

The conclusion of the theorem then follows from combining the aforementioned convergence for the 
infinite volume field, with the harmonic coupling Theorem \ref{decouple} (which gives $%
\sum_{x\in D_{N}}\rho _{N}\left( x\right) \phi ^{0}\left( x\right)
=\sum_{x\in D_{N}}\rho _{N}\left( x\right) \phi ^{D_{N},0}\left( x\right) $
on the good event). The argument was detailed out in \cite{M}, Section 7
(see also the proof of Lemma \ref{jasonerr} above).
\end{proof}

For the rest of the paper we will only apply the convergence of the second
moment (i.e., $k=1$ result) in Theorem \ref{thm: miller CLT}.

\section{Harmonic averages\label{harmonic}}

Our method to prove Proposition \ref{thm: tail bound} is built upon Theorem %
\ref{decouple} and a detailed study of the harmonic average of the
Ginzburg-Landau field. Given $B\subset \mathbb{Z}^{2}$, $v\in B$ and $y\in
\partial B$, we denote by $a_{B}\left( v,\cdot \right) $ the harmonic
measure on $\partial B$ seen from $v$. In other words, let $S^{x}$ denote
the simple random walk starting at $x$, and $\tau _{\partial B}=\inf \left\{
t>0:S\left[ t\right] \in \partial B\right\} $, we have 
\begin{equation*}
a_{B}\left( x,y\right) =\mathbb{P}\left( S^{x}\left[ \tau _{\partial B}%
\right] =y\right) .
\end{equation*}

Given $v\in \mathbb{Z}^{2}$ and $R>r>0$, let $B_{R}\left( v\right) =\left\{
y\in \mathbb{Z}^{2}:\left\vert v_{1}-y_{1}\right\vert ^{2}+\left\vert
v_{2}-y_{2}\right\vert ^{2}<R^{2}\right\} $, and $A_{r,R}\left( v\right)
:=B_{R}\left( v\right) \setminus B_{r}\left( v\right) $. Define the circle
average of the Ginzburg-Landau field with radius $R$ at $v$ by 
\begin{equation}
C_{R}\left( v,\phi \right) =\sum_{y\in \partial B_{R}\left( v\right)
}a_{B_{R}\left( v\right) }\left( v,y\right) \phi \left( y\right) .
\label{CR}
\end{equation}

For each $\varepsilon ,R>0$, such that $(1+\eps)R< \text{dist}(v,\partial D_N)$, we take a non-negative smooth radial function $%
f_{\varepsilon } \in C_c^\infty([1-\eps, 1+\eps])$ such that 
$f_{\varepsilon }\left( 1-s\right) =f_{\varepsilon }\left( 1+s\right) $ for $%
s\in \left[ 0,\varepsilon \right] $ and $\int_{1-\varepsilon
}^{1+\varepsilon }f_{\varepsilon }\left( s\right) ds=1$. We further define%
\begin{equation}
X_{R}\left( v,\phi \right) =\sum_{r=\left( 1-\varepsilon \right) R}^{\left(
1+\varepsilon \right) R}f_{\varepsilon }\left( r/R\right) C_{r}\left( v,\phi
\right) .  \label{X}
\end{equation}

The crucial object that we use below is the increment of the harmonic
average process $X$. For $v\in D_{N}$, $\left( 1+\varepsilon \right) ^{-1}$%
dist$\left( v,\partial D_{N}\right) >R_{1}>R_{2}>0$, we would like to study
the increment 
\begin{equation*}
X_{R_{2}}\left( v,\phi \right) -X_{R_{1}}\left( v,\phi \right) =\left(
\sum_{r=\left( 1-\varepsilon \right) R_{2}}^{\left( 1+\varepsilon \right)
R_{2}}f_{\varepsilon }\left( r/R_{2}\right) -\sum_{r=\left( 1-\varepsilon
\right) R_{1}}^{\left( 1+\varepsilon \right) R_{1}}f_{\varepsilon }\left(
r/R_{1}\right) \right) \sum_{y\in \partial B_{r}\left( v\right)
}a_{B_{r}\left( v\right) }\left( v,y\right) \phi \left( y\right) .
\end{equation*}%
This can be written as $\sum_{y\in D_{N}}\rho_N \left( v,y\right) \phi \left(
y\right) $, where we define%
\begin{equation*}
\rho_N \left( v,y\right) =\left[ f_{\varepsilon }\left( \frac{\left\vert
v-y\right\vert }{R_{2}}\right) -f_{\varepsilon }\left( \frac{\left\vert
v-y\right\vert }{R_{1}}\right) \right] a_{B_{\left\vert v-y\right\vert
}\left( v\right) }\left( v,y\right) .
\end{equation*}%
The definition of $\rho_N $ depends on $N, R_1, R_2$, as the
definition of the harmonic measure depends on the lattice spacing. We will
omit the $R_1, R_2$ dependence for the simplicity of notations, and also the $v$ dependence if it is clear from the context. 

\begin{lemma}
\label{lem:harmtest}For any discrete harmonic function $h$ in $D_{N}$, we
have $\sum_{y\in D_{N}}\rho_N \left( v,y\right) h\left( y\right) =0.$
\end{lemma}

\begin{proof}
Suppose $h$ is define up to $\partial D_{N}$, and $h|_{\partial D_{N}}=H$.
We conclude the proof by showing for $i=1,2$%
\begin{equation*}
\sum_{r=\left( 1-\varepsilon \right) R_{i}}^{\left( 1+\varepsilon \right)
R_{i}}f_{\varepsilon }\left( r/R_{i}\right) \sum_{y\in \partial B_{r}\left(
v\right) }a_{B_{r}\left( v\right) }\left( v,y\right) h\left( y\right)
=h\left( v\right) .
\end{equation*}%
Indeed, since $h$ is harmonic,%
\begin{equation*}
h\left( y\right) =\sum_{z\in \partial D_{N}}a_{D_{N}}\left( y,z\right)
H\left( z\right) .
\end{equation*}%
Using the fact that 
\begin{equation*}
\sum_{y\in \partial B_{r}\left( v\right) }a_{B_{r}\left( v\right) }\left(
v,y\right) a_{D_{N}}\left( y,z\right) =a_{D_{N}}\left( v,z\right) ,
\end{equation*}%
we obtain%
\begin{equation*}
\sum_{r=\left( 1-\varepsilon \right) R_{i}}^{\left( 1+\varepsilon \right)
R_{i}}f_{\varepsilon }\left( r/R_{i}\right) \sum_{y\in \partial B_{r}\left(
v\right) }a_{B_{r}\left( v\right) }\left( v,y\right) h\left( y\right)
=\sum_{r=\left( 1-\varepsilon \right) R_{i}}^{\left( 1+\varepsilon \right)
R_{i}}f_{\varepsilon }\left( r/R_{i}\right) \sum_{z\in \partial
D_{N}}a_{D_{N}}\left( v,z\right) H\left( z\right) =h\left( v\right) .
\end{equation*}
\end{proof}

The following result is a consequence of Theorem \ref{decouple} and the
lemma above.

\begin{lemma}
\label{average}Suppose the same conditions in Theorem \ref{decouple} holds.
Given $v\in D_{N}$, $R_{1}>R_{2}>0$, $\varepsilon >0$ such that $\left(
1+2\varepsilon \right) R_{1}<$dist$\left( v,\partial D_{N}\right) $, $\left(
1+2\varepsilon \right) R_{2}<\left( 1-2\varepsilon \right) R_{1}$. Let $%
\delta $ be the constant from Theorem \ref{decouple}. Let $\phi ^{f}$ be
sampled from Ginzburg-Landau field (\ref{GLD}), and $\phi ^{0}$ be sampled
from the zero boundary Ginzburg-Landau field on $D_{N}$. Then, on an event
with probability $1-O\left( R_{1}^{-\delta }\right) $, we have 
\begin{equation*}
X_{R_{2}}\left( v,\phi ^{f}\right) -X_{R_{1}}\left( v,\phi ^{f}\right)
=X_{R_{2}}\left( v,\phi ^{0}\right) -X_{R_{1}}\left( v,\phi ^{0}\right) .
\end{equation*}
\end{lemma}

We sometimes omit the dependence of $X$ on $v$ and $\phi $ when it is clear
from the context.

We are mostly concerned with large deviation estimates, and therefore with
moment generating functions. Thus we will use Proposition \ref{mgfCLT}
below, which gives a Gaussian limit of the moment generating function of
macroscopic observables.

Now fix $v=0$. Given $\varepsilon >0$ fixed in the definition (\ref{X}) and $r>0$, take $%
\varepsilon _{1}=\varepsilon ^{1/4}$, and note that we can write 
\begin{eqnarray*}
X_{\left( 1+\varepsilon _{1}\right) r}\left(0,\phi \right) &=&\sum_{y\in
D_{N}}\rho _{r,+}\left( y\right) \phi \left( y\right) , \\
X_{\left( 1-\varepsilon _{1}\right) r}\left( 0,\phi \right) &=&\sum_{y\in
D_{N}}\rho _{r,-}\left( y\right) \phi \left( y\right) .
\end{eqnarray*}%
Let $A_{r_{1},r_{2}} =B_{r_{2}}\left( 0\right)
\setminus B_{r_{1}}\left( 0\right) $. Note that $\rho _{r,+}$ and $\rho
_{r,-}$ are supported on annuli $A_{\left( 1+\varepsilon _{1}-\varepsilon
\right) r,\left( 1+\varepsilon _{1}+\varepsilon \right) r}$ and $A_{\left(
1-\varepsilon _{1}-\varepsilon \right) r,\left( 1-\varepsilon
_{1}+\varepsilon \right) r}$ respectively, and let $\rho_r =\rho _{r,+}-\rho
_{r,-}$. We further notice that as $r\rightarrow \infty $, the rescaled
harmonic measure 
\begin{equation*}
ra_{B_{r}\left( 0\right) }\left( 0,\cdot \right) \rightarrow 1/2\pi .
\end{equation*}%
Thus as $r\rightarrow \infty $ and $y/r\rightarrow x$, $r\rho _{r,+},r\rho
_{r,-}$ converge respectively to the smooth functions 
\begin{eqnarray*}
f^{+}\left( x\right) &=&\frac{1}{2\pi \left\vert x\right\vert }%
f_{\varepsilon }\left( \left\vert x\right\vert \right) ,\text{ }x\in
A_{\left( 1+\varepsilon _{1}-\varepsilon \right) ,\left( 1+\varepsilon
_{1}+\varepsilon \right) }, \\
f^{-}\left( x\right) &=&\frac{1}{2\pi \left\vert x\right\vert }%
f_{\varepsilon }\left( \left\vert x\right\vert \right) ,\text{ }x\in
A_{\left( 1-\varepsilon _{1}-\varepsilon \right) ,\left( 1-\varepsilon
_{1}+\varepsilon \right) }.
\end{eqnarray*}%
To simplify the notations below, we write $A_{+}:=A_{\left( 1+\varepsilon
_{1}-\varepsilon \right) ,\left( 1+\varepsilon _{1}+\varepsilon \right) }$
and $A_{-}:=A_{\left( 1-\varepsilon _{1}-\varepsilon \right) ,\left(
1-\varepsilon _{1}+\varepsilon \right) }$. We further denote $f=f^{+}-f^{-}$%
. The following estimate is proved by combining Theorem \ref{thm: miller CLT}
with the Brascamp-Lieb inequality.

\begin{proposition}
\label{mgfCLT}Let $D=\left[ -1,1\right] ^{2}$and $v=0$. For any $\varepsilon
_{1}>0$ small enough, $t>0$ and $r=r\left( N\right) $ such that $N/4<\left(
1-\varepsilon _{1}\right) r<\left( 1+\varepsilon _{1}\right) r<N$, we have
that 
\begin{eqnarray}
&&\log \mathbb{E}^{D_{N},0}\left[ \exp \left( t\left( X_{\left(
1-\varepsilon _{1}\right) r}-X_{\left( 1+\varepsilon _{1}\right) r}\right)
\right) \right]  \notag \\
&=&\frac{t^{2}}{2}\int_{A_{+}\cup A_{-}}f\left( x\right) g_{a,D}\left(
x,y\right) f\left( y\right) dxdy+f_{1}\left( \varepsilon _{1},r\right)
t^{2}+f_{2}\left( \varepsilon _{1}\right) t^{4},  \label{logmomint}
\end{eqnarray}%
where $a\left( V\right) =\bar{a}\left( V\right) I$ is defined in Theorem \ref%
{thm: miller CLT}, $g_{a,D}\left( x,\cdot \right) $ is the Dirichlet Green's
function that solves the PDE%
\begin{equation*}
\left\{ 
\begin{array}{cc}
\nabla ^{\ast }\cdot a\nabla u=\delta \left( x\right) & \text{ in }D \\ 
u=0 & \text{ on }\partial D%
\end{array}%
\right. .
\end{equation*}%
And there exists $C<\infty $, such that $\left\vert f_{2}\left( \varepsilon
_{1}\right) \right\vert \leq C\varepsilon _{1}^{2}$, and $f_{1}\left(
\varepsilon _{1},r\right) /\varepsilon _{1}\rightarrow 0$ as $N\rightarrow
\infty $. Moreover, there exists $g=g\left( V\right) $, such that 
\begin{equation}
\log \mathbb{E}^{D_{N},0}\left[ \exp \left( t\left( X_{\left( 1-\varepsilon
_{1}\right) r}-X_{\left( 1+\varepsilon _{1}\right) r}\right) \right) \right]
=\frac{t^{2}}{2}g\log \frac{1+\varepsilon _{1}}{1-\varepsilon _{1}}+\hat{f}%
_{1}\left( \varepsilon _{1},r\right) t^{2}+f_{2}\left( \varepsilon
_{1}\right) \left( t^{2}+t^{4}\right) ,  \label{logmomratio}
\end{equation}%
where $\hat{f}_{1}\left( \varepsilon _{1},r\right) /\varepsilon
_{1}\rightarrow 0$ as $N\rightarrow \infty $.
\end{proposition}

\begin{proof}
We first show 
\begin{equation}
\mathbb{E}^{D_{N},0}\left[ \exp \left( t\left( X_{\left( 1-\varepsilon
_{1}\right) r}-X_{\left( 1+\varepsilon _{1}\right) r}\right) \right) \right]
=\frac{t^{2}}{2}\text{Var}^{D_{N},0}\left[ X_{\left( 1-\varepsilon
_{1}\right) r}-X_{\left( 1+\varepsilon _{1}\right) r}\right] +f_{2}\left(
\varepsilon _{1}\right) t^{4}.  \label{taylor}
\end{equation}%
Indeed, we can expand $\mathbb{E}^{D_{N},0}\left[ \exp \left( t\left(
X_{\left( 1-\varepsilon _{1}\right) r}-X_{\left( 1+\varepsilon _{1}\right)
r}\right) \right) \right] $ into Taylor series of $t$, and bound the higher
moments. Use the fact that the distribution of $\phi $ is symmetric, we can
write%
\begin{eqnarray*}
&&\mathbb{E}^{D_{N},0}\left[ \exp \left( t\left( X_{\left( 1-\varepsilon
_{1}\right) r}-X_{\left( 1+\varepsilon _{1}\right) r}\right) \right) \right]
\\
&=&\mathbb{E}^{D_{N},0}\left[ \exp \left( t\sum_{x\in D_{N}}\phi \left(
x\right) \rho_r \left( x\right) \right) \right] \\
&=&1+\frac{t^{2}}{2}\text{Var}^{D_{N},0}\left[ \sum_{x\in D_{N}}\phi \left(
x\right) \rho_r \left( x\right) \right] +\sum_{k=2}^{\infty }\frac{t^{2k}}{%
\left( 2k\right) !}\mathbb{E}^{D_{N},0}\left\vert \sum_{x\in D_{N}}\phi
\left( x\right) \rho_r \left( x\right) \right\vert ^{2k}.
\end{eqnarray*}%
We now claim 
\begin{equation}
\sum_{k=2}^{\infty }\frac{t^{2k}}{\left( 2k\right) !}\mathbb{E}%
^{D_{N},0}\left\vert \sum_{x\in D_{N}}\phi \left( x\right) \rho_r \left(
x\right) \right\vert ^{2k}=O\left( \varepsilon _{1}^{2}\right) t^{4}.
\label{high}
\end{equation}%
By the Brascamp-Lieb inequality for even moments 
\eqref{eq: BL even moment
bound}, we have 
\begin{equation}
\mathbb{E}^{D_{N},0}\left\vert \sum_{x\in D_{N}}\phi \left( x\right) \rho_r
\left( x\right) \right\vert ^{2k}\leq c_{-}^{-k}\mathbb{E}%
_{DGFF}^{D_{N,}0}\left\vert \sum_{x\in D_{N}}\phi \left( x\right) \rho_r
\left( x\right) \right\vert ^{2k}\leq \left( 2k-1\right)
!!c_{-}^{-k}\varepsilon _{1}^{2k}.  
\end{equation}%
By taking $\varepsilon _{1}$ small enough such that $\varepsilon _{1}t^{2}<1$%
, summing over $k$ yields (\ref{high}), and thus concludes (\ref{taylor}).

To prove (\ref{logmomint}), it suffices to obtain the asymptotic variance of 
$X_{\left( 1-\varepsilon _{1}\right) r}-X_{\left( 1+\varepsilon _{1}\right)
r}$. The Brascamp-Lieb inequality implies Var$^{D_{N},0}\left[ X_{\left(
1-\varepsilon _{1}\right) r}-X_{\left( 1+\varepsilon _{1}\right) r}\right]
\leq C\log \frac{1+\varepsilon _{1}}{1-\varepsilon _{1}}$ for all $N\geq 1$.
Notice that from Lemma \ref{lem:harmtest} and standard harmonic measure
estimates (see e.g., \cite{LL}) $\left\vert ra_{B_{r}\left( 0\right) }\left(
0,\cdot \right) -\frac{1}{2\pi }\right\vert =O\left( 1/r\right) $, we see
that the spatial average $\sum_{x\in D_{N}}\phi \left( x\right) \rho_r \left(
x\right) =r^{-1}\sum_{x\in D_{N}}\phi \left( x\right) r\rho_r \left( x\right) $%
\ satisfies the conditions of Theorem \ref{thm: miller CLT}. Apply Theorem %
\ref{thm: miller CLT}, and note that $r\rho_r \left( x\right) \rightarrow
f\left( x\right) $ as $r\rightarrow \infty $, we see that\ there exists a
positive definite $2\times 2$ matrix $a\left( V\right) =\bar{a}\left(
V\right) I$, such that 
\begin{gather*}
\text{Var}^{D_{N},0}\left[ X_{\left( 1-\varepsilon _{1}\right) r}-X_{\left(
1+\varepsilon _{1}\right) r}\right] =\text{Var}^{D_{N},0}\left[
r^{-1}\sum_{x\in D_{N}}\phi \left( x\right) r\rho_r \left( x\right) \right] \\
=\text{Var}_{a-GFF}^{D}\left[ \int_{A_{+}}f\left( x\right) h\left( x\right)
dx+\int_{A_{-}}f\left( x\right) h\left( x\right) dx\right] +f_{1}\left(
\varepsilon _{1},r\right) ,
\end{gather*}%
where $f_{1}\left( \varepsilon _{1},r\right) /\varepsilon _{1}\rightarrow 0$
as $N\rightarrow \infty $. Then the definition of the $a$-GFF implies 
\begin{equation}
\text{Var}_{a-GFF}^{D}\left[ \int_{A_{+}}f\left( x\right) h\left( x\right)
dx+\int_{A_{-}}f\left( x\right) h\left( x\right) dx\right] =\int_{A_{+}\cup
A_{-}}f\left( x\right) g_{a,D}\left( x,y\right) f\left( y\right) dxdy,
\label{computevar}
\end{equation}%
which concludes (\ref{logmomint}).

To obtain (\ref{logmomratio}), we further claim that there exists $g=g\left( 
\bar{a}\right) $, such that 
\begin{equation}
\text{Var}^{D_{N},0}\left[ X_{\left( 1-\varepsilon _{1}\right) r}-X_{\left(
1+\varepsilon _{1}\right) r}\right] =g\log \frac{1+\varepsilon _{1}}{%
1-\varepsilon _{1}}+\hat{f}_{1}\left( \varepsilon _{1},r\right) + O(\eps_1^2) ,  \label{varatio}
\end{equation}%
This can be proved by an explicit evaluation of the integral (\ref%
{computevar}). Instead, we give a proof here using comparison to the
standard discrete GFF. Since $g_{a,D}\left( x,y\right) =\bar{a}^{-1}\Delta
_{D}^{-1}\left( x,y\right) $, where $\Delta _{D}$ is the standard Dirichlet
Laplacian in $D$, we can conclude (\ref{varatio}) by showing 
\begin{equation}
\text{Var}_{DGFF}^{D_{N},0}\left[ X_{\left( 1-\varepsilon _{1}\right)
r}-X_{\left( 1+\varepsilon _{1}\right) r}\right] =\frac{2}{\pi }\log \frac{%
1+\varepsilon _{1}}{1-\varepsilon _{1}}+O\left( \varepsilon _{1}^{2}\right) ,
\label{gffdiff}
\end{equation}%
since the left side converge as $N\rightarrow \infty $ to $\int_{A_{+}\cup
A_{-}}f\left( x\right) \Delta _{D}^{-1}\left( x,y\right) f\left( y\right)
dxdy$, which only differs from (\ref{computevar}) by a multiplicative
constant. This then follows from an explicit computation: let $R=\left(
1+\varepsilon _{1}+\varepsilon _{1}^{4}\right) r$, using the Gibbs-Markov
property of discrete GFF, we have 
\begin{equation*}
\text{Var}_{DGFF}^{D_{N},0}\left[ X_{\left( 1-\varepsilon _{1}\right)
r}-X_{\left( 1+\varepsilon _{1}\right) r}\right] =\text{Var}_{DGFF}^{D_{R},0}%
\left[ X_{\left( 1-\varepsilon _{1}\right) r}-X_{\left( 1+\varepsilon
_{1}\right) r}\right] .
\end{equation*}%
Since Var$_{DGFF}^{D_{R},0}\left[ X_{\left( 1+\varepsilon _{1}\right) r}%
\right] \leq C\varepsilon _{1}^{4}$, the right side equals to%
\begin{eqnarray*}
&&\text{Var}_{DGFF}^{D_{R},0}\left[ X_{\left( 1-\varepsilon _{1}\right) r}%
\right] +\text{Cov}_{DGFF}^{D_{R},0}\left[ X_{\left( 1-\varepsilon
_{1}\right) r},X_{\left( 1+\varepsilon _{1}\right) r}\right] +O\left(
\varepsilon _{1}^{4}\right) \\
&=&\text{Var}_{DGFF}^{D_{R},0}\left[ X_{\left( 1-\varepsilon _{1}\right) r}%
\right] +O\left( \varepsilon _{1}^{2}\right) ,
\end{eqnarray*}%
where we apply Cauchy-Schwarz to bound the covariance. To compute Var$%
_{DGFF}^{D_{R},0}\left[ X_{\left( 1-\varepsilon _{1}\right) r}\right] $,
again using the Gibbs-Markov property, which implies for any $N/4<r<N$, 
\begin{eqnarray*}
\text{Var}_{DGFF}^{D_{R},0}\left[ C_{r}\left( 0,\phi \right) \right] &=&%
\text{Var}_{DGFF}^{D_{R},0}\left[ \phi \left( 0\right) \right] -\text{Var}%
_{DGFF}^{D_{r},0}\left[ \phi \left( 0\right) \right] \\
&=&\frac{2}{\pi }\log \frac{R}{r}+O\left( 1/N\right) .
\end{eqnarray*}%
Here we applied the standard Green's function asymptotics (see, e.g., \cite%
{LL}) to obtain the last line. Take a weighted sum over $f_{\varepsilon }$
(with $\varepsilon =\varepsilon _{1}^{4}$) we have 
\begin{multline}
\text{Var}_{DGFF}^{D_{R},0}\left[ X_{\left( 1-\varepsilon _{1}\right) r}%
\right] =\text{Var}_{DGFF}^{D_{R},0}\big[ C_{\left( 1-\varepsilon
_{1}\right) r}\left( 0,\phi \right) \\ -\sum_{r_{1}=-\varepsilon
r}^{\varepsilon r}f_{\varepsilon }\left( 1+\frac{r_{1}}{\left( 1-\varepsilon
_{1}\right) r}\right) \left( C_{\left( 1-\varepsilon _{1}\right) r}\left(
0,\phi \right) -C_{\left( 1-\varepsilon _{1}\right) r+r_{1}}\left( 0,\phi
\right) \right) \big] .  \label{localcomp}
\end{multline}%
Again, the Gibbs-Markov property implies for any $N/4<r_{1}<r_{2}<R,$%
\begin{equation*}
\text{Var}_{DGFF}^{D_{R},0}\left[ C_{r_{1}}\left( 0,\phi \right)
-C_{r_{2}}\left( 0,\phi \right) \right] =\text{Var}_{DGFF}^{D_{r_{2}},0}%
\left[ C_{r_{1}}\left( 0,\phi \right) \right] =\frac{2}{\pi }\log \frac{r_{2}%
}{r_{1}}+O\left( 1/N\right) .
\end{equation*}%
Substitute into the right side of (\ref{localcomp}), we conclude that 
\begin{equation*}
\text{Var}_{DGFF}^{D_{R},0}\left[ X_{\left( 1-\varepsilon _{1}\right) r}%
\right] =\text{Var}_{DGFF}^{D_{R},0}\left[ C_{\left( 1-\varepsilon
_{1}\right) r}\left( 0,\phi \right) \right] +O\left( \varepsilon
_{1}^{2}\right) =\frac{2}{\pi }\log \frac{1+\varepsilon _{1}}{1-\varepsilon
_{1}}+O\left( \varepsilon _{1}^{2}\right) .
\end{equation*}%
This yields (\ref{varatio}).
\end{proof}

\section{Pointwise distribution for Ginzburg-Landau field\label{ptwise}}

The main result of this section is the Gaussian tail for the Ginzburg-Landau
field at one site (Proposition \ref{thm: tail bound}). To prove this we will
employ a multiscale decomposition argument to obtain the approximate
Gaussian asymptotics of moment generating function of the harmonic average
process.

We first introduce the {proper} scales in order to carry out the inductive
argument. Given any $v\in D_{N}$, $\varepsilon >0$ and $c\in \left(
0,1\right) $, denote by $\Delta =$dist$\left( v,\partial D_{N}\right) $ and $%
M=M\left( c\right) =\left( 1-c\right) \log \Delta /\log \left( 1+\varepsilon
\right) $. Define the sequence of numbers $\left\{ r_{k}\right\}
_{k=1}^{\infty }$, $\left\{ r_{k,+}\right\} _{k=0}^{\infty }$ and $\left\{
r_{k,-}\right\} _{k=0}^{\infty }$ by%
\begin{eqnarray}
r_{k} &=&\left( 1+\varepsilon \right) ^{-k}\Delta ,  \label{scale} \\
r_{k,+} &=&\left( 1+\varepsilon ^{3}\right) r_{k},  \notag \\
r_{k,-} &=&\left( 1-\varepsilon ^{3}\right) r_{k}.  \notag
\end{eqnarray}%
We also define%
\begin{eqnarray*}
X_{r_{k,+}}\left( v\right) &=&\sum_{r=\left( 1-\varepsilon ^{4}\right)
r_{k,+}}^{\left( 1+\varepsilon ^{4}\right) r_{k,+}}f_{\varepsilon
^{4}}\left( \frac{r}{r_{k,+}}\right) C_{r}\left( v,\phi \right) , \\
X_{r_{k,-}}\left( v\right) &=&\sum_{r=\left( 1-\varepsilon ^{4}\right)
r_{k,-}}^{\left( 1+\varepsilon ^{4}\right) r_{k,-}}f_{\varepsilon
^{4}}\left( \frac{r}{r_{k,-}}\right) C_{r}\left( v,\phi \right) ,
\end{eqnarray*}%
where $C_{r}$ is defined in (\ref{CR}), and $f_{\varepsilon ^{4}}$ is the
smooth function defined just below (\ref{CR}).

For $r>0$ denote by $\mathbb{P}^{r,0}$ the law of the Ginzburg-Landau field
in $B_{r}\left( v\right) $ with zero boundary condition (and denote by $%
\mathbb{E}^{r,0}$ the corresponding expectation). The basic building block
of all our large deviation estimates is the following.

\begin{theorem}
\label{Gauss}There exists $g=g(V)$, such that given $C>0$, $c\in \left( 0,1\right) $ we have for all $v\in
D_{N}$ and $\left\vert t\right\vert \leq C$,%
\begin{equation*}
\log \mathbb{E}^{D_{N},0}\left[ \exp \left( tX_{r_{M,+}}\left( v\right)
\right) \right] =\frac{t^{2}}{2}\left( 1-c\right) g\log \Delta +o_{\Delta
}\left( \log \Delta \right)( t^{2}+t^4)+O\left( 1\right) ,
\end{equation*}%
where the $O\left( 1\right) $ term depends on $C$ and $c$.
\end{theorem}

\begin{remark}
\label{Gauss2}The proof of Theorem \ref{Gauss} also yields%
\begin{equation*}
\log \mathbb{E}^{D_{N},0}\left[ \exp \left( tX_{r_{M,+}}\left( v\right)
-tX_{r_{0,-}}\left( v\right) \right) \right] =\frac{t^{2}}{2}\left(
1-c\right) g\log \Delta +o_{\Delta }\left( \log \Delta \right) (t^{2}+t^4)+O\left(
1\right) .
\end{equation*}%
This will be used to prove Theorem \ref{expcircle} below.
\end{remark}

Roughly speaking, this theorem indicates that as long as the last scale $%
r_{M}$ satisfies $r_{M}>\Delta ^{c}$, for some $c>0$, the harmonic average $%
X_{r_{M,+}}$ is nearly Gaussian with mean zero and variance $g\log \frac{N}{%
r_{M}}$. To prove this theorem we will first prove the following decoupling
result. We denote 
\begin{eqnarray*}
W_{j} &=&\exp \left( t\left( X_{r_{j,+}}-X_{r_{j-1,-}}\right) \right) , \\
Y_{j} &=&\exp \left( t\left( X_{r_{j,-}}-X_{r_{j,+}}\right) \right) , \\
Z_{j} &=&\exp \left( t\left( X_{r_{j,+}}\right) \right) .
\end{eqnarray*}%
Here $W_{j}$ encodes the distribution of the increment of the harmonic
average process $X_{\cdot }$, $Y_{j}$ are introduced to make the $W_{j}$'s
decouple, and we will show they have little influence on the large deviation
estimates.

\begin{theorem}
\label{ind}Given $C>0$, $c\in \left( 0,1\right) $ and $C_{1}<\infty $ we
have for all $v\in D_{N}$ and $\left\vert t\right\vert \leq C$,%
\begin{equation*}
\log \mathbb{E}^{D_{N},0}\left[ \exp \left( tX_{r_{M,+}}\right) \right]
=\sum_{j=1}^{M}\log \mathbb{E}^{r_{j-1},0}\left[ W_{j}\right] +\log \mathbb{E%
}^{D_{N},0}\left[ \exp \left( tX_{r_{0,-}}\right) \right] +t^{2}f\left(
\varepsilon \right) \log \Delta +O\left( 1\right) ,
\end{equation*}%
where $\left\vert f\left( \varepsilon \right) \right\vert \leq
C_{1}\varepsilon ^{2}/\log \left( 1+\varepsilon \right) $, and the $O\left(
1\right) $ term depends on $C$ and constants from Lemma \ref{decouple}. More
precisely, we have for and $k=1,...,M$,%
\begin{eqnarray}
&&\log \mathbb{E}^{D_{N},0}\left[ \exp \left( tX_{r_{k,+}}\right) \right]
\label{ind2} \\
&=&\sum_{j=1}^{k}\log \mathbb{E}^{r_{j-1},0}\left[ W_{j}\right] +\log 
\mathbb{E}^{D_{N},0}\left[ \exp \left( tX_{r_{0,-}}\right) \right]
+t^{2}O\left( \frac{\varepsilon ^{2}}{\log \left( 1+\varepsilon \right) }%
\right) \log \frac{\Delta }{r_{k}}+O\left( \sum_{j=1}^{k-1}r_{j}^{-\delta
}\right) .  \notag
\end{eqnarray}
\end{theorem}

Notice that $\sum_{j=1}^{k}r_{j}^{-\delta }$ is a geometric sum, and is thus 
$O\left( r_{k}^{-\delta }\right) $.

\begin{proof}[Proof of Theorem \protect\ref{Gauss}]
Applying Proposition \ref{mgfCLT} (in particular, (\ref{logmomratio})), we
see that there exists $g=g\left( V\right) $, such that as $\Delta
\rightarrow \infty $, 
\begin{equation*}
\log \mathbb{E}^{r_{j-1},0}\left[ W_{j}\right] =\frac{t^{2}}{2}g\log \frac{%
r_{j-1}}{r_{j}}+o_{\Delta }\left( 1\right) \log \frac{r_{j-1}}{r_{j}}%
t^{2}+O\left( \varepsilon ^{2}\right) \left( t^{2}+t^{4}\right) .
\end{equation*}%
Summing over $j$ and applying Theorem \ref{ind}, we have 
\begin{eqnarray*}
\log \mathbb{E}^{D_{N},0}\left[ \exp \left( tX_{r_{M,+}}\right) \right] &=&%
\frac{t^{2}}{2}\left( 1-c\right) g\log \Delta +o_{\Delta }\left( \log \Delta
\right) t^{2}+\left( t^{2}+t^{4}\right) O\left( \frac{\varepsilon ^{2}}{\log
\left( 1+\varepsilon \right) }\right) \log \Delta \\
&&+t^{2}f\left( \varepsilon \right) \log \Delta +O\left( 1\right) .
\end{eqnarray*}%
Since $|t|\leq C$, sending $\varepsilon \rightarrow 0$ we conclude Theorem %
\ref{Gauss}.
\end{proof}

\subsection{Proof of Theorem \protect\ref{ind}\label{indpf}}

We write $X_{r_{M,+}}$ as a telescoping sum%
\begin{equation*}
X_{r_{M,+}}=\left( X_{r_{M,+}}-X_{r_{M-1,-}}\right) +\left(
X_{r_{M-1,-}}-X_{r_{M-2},+}\right) +...\left( X_{r_{1,+}}-X_{r_{0,-}}\right)
+X_{r_{0,-}},
\end{equation*}%
and therefore%
\begin{eqnarray*}
Z_{M} &=&e^{tX_{r_{M,+}}}=e^{tX_{r_{0,-}}}\prod_{j=1}^{M}\exp \left( t\left(
X_{r_{j,+}}-X_{r_{j-1,-}}\right) \right) \prod_{j=1}^{M-1}\exp \left(
t\left( X_{r_{j,-}}-X_{r_{j,+}}\right) \right) \\
&=&W_{M}Y_{M-1}W_{M-1}...Y_{1}W_{1}\exp \left( tX_{r_{0,-}}\right) .
\end{eqnarray*}%
Notice that 
\begin{equation*}
Z_{k}=W_{k}Y_{k-1}Z_{k-1}=W_{k}Z_{k-1}+W_{k}\left( Y_{k-1}-1\right) Z_{k-1}.
\end{equation*}%
Since $Z_{k-1}=W_{k-1}Y_{k-2}Z_{k-2}$, by iterating we obtain 
\begin{eqnarray}
Z_{k} &=&\sum_{m=1}^{k-1}W_{m+1}Z_{m}\prod_{j=m+2}^{k}\left( W_{j}\left(
Y_{j-1}-1\right) \right) +Z_{1}\prod_{j=2}^{k}W_{j}\left( Y_{j-1}-1\right)
\label{sum} \\
&=&W_{k}Z_{k-1}+E_{Y}^{\left( k\right) },  \notag
\end{eqnarray}%
where 
\begin{equation}
E_{Y}^{\left( k\right) }=\sum_{m=1}^{k-2}W_{m+1}Z_{m}\prod_{j=m+2}^{k}\left(
W_{j}\left( Y_{j-1}-1\right) \right) +Z_{1}\prod_{j=2}^{k}W_{j}\left(
Y_{j-1}-1\right) .  \label{eyk}
\end{equation}

We will show that the main contribution to $\log \mathbb{E}^{D_{N},0}\left[
Z_{k}\right] $ is the first term in the summation (\ref{sum}), i.e., $\log 
\mathbb{E}^{D_{N},0}\left[ W_{k}Z_{k-1}\right] $, and that the other terms
are negligible. We denote by $\mathcal{F}_{k}=\sigma \left( \phi \left(
x\right) :x\in D_{N}\setminus B_{r_{k}}\left( v\right) \right) $, and take $%
\mathcal{G}=\left\{ \max_{x\in B_{\Delta }\left( v\right) }\left\vert \phi
\left( x\right) \right\vert \leq \left( \log \Delta \right) ^{2}\right\} $.
Recall that by Lemma \ref{bad}, $\mathbb{P}\left( \mathcal{G}^{c}\right)
\leq \exp \left( -c_{1}\left( \log \Delta \right) ^{3}\right) $ for some $%
c_{1}>0$.

We can write 
\begin{equation*}
\mathbb{E}^{D_{N},0}\left[ Z_{k}\right] =\mathbb{E}^{D_{N},0}\left[ Z_{k}1_{%
\mathcal{G}}\right] +\mathbb{E}^{D_{N},0}\left[ Z_{k}1_{\mathcal{G}^{c}}%
\right] .
\end{equation*}%
Since $|t|\leq C$, apply H\"{o}lder and the exponential Brascamp-Lieb
inequality,%
\begin{eqnarray}
\mathbb{E}^{D_{N},0}\left[ Z_{k}1_{\mathcal{G}^{c}}\right] &\leq &\left( 
\mathbb{E}^{D_{N},0}\left[ Z_{k}^{2}\right] \right) ^{1/2}\mathbb{P}%
^{D_{N},0}\left( \mathcal{G}^{c}\right) ^{1/2}  \notag \\
&\leq &\exp \left( 2c_{-}^{-1}t^{2}\text{Var}_{DGFF}^{D_{N},0}\left[
X_{r_{M,+}}\right] \right) \mathbb{P}^{D_{N},0}\left( \mathcal{G}^{c}\right)
^{1/2}  \notag \\
&\leq &\exp \left( Ct^{2}\log \Delta -\frac{c_{1}}{2}\left( \log \Delta
\right) ^{3}\right) \leq \exp \left( -\frac{c_{1}}{4}\left( \log \Delta
\right) ^{3}\right) ,  \label{Zbad}
\end{eqnarray}%
which is negligible.

In order to prove (\ref{ind2}), we set up a joint induction for

\begin{itemize}
\item There exists an absolute constant $C_{1}<\infty $, such that for all $%
k\geq 0$, 
\begin{eqnarray}
&&\log \mathbb{E}^{D_{N},0}\left[ Z_{k}1_{\mathcal{G}}\right]  \label{Zgood}
\\
&=&\sum_{j=1}^{k}\log \mathbb{E}^{r_{j-1},0}\left[ W_{j}1_{\mathcal{G}}%
\right] +\log \mathbb{E}^{D_{N},0}\left[ \exp \left( tX_{r_{0,-}}\right) 1_{%
\mathcal{G}}\right] +t^{2}F_{k}+R_{k-1},  \notag
\end{eqnarray}%
where $\left\vert F_{k}\right\vert \leq C_{1}k\varepsilon ^{2}=C_{1}\frac{%
\varepsilon ^{2}}{\log \left( 1+\varepsilon \right) }\log \frac{\Delta }{%
r_{k}}$, and $\left\vert R_{k-1}\right\vert \leq
C_{1}\sum_{j=1}^{k-1}r_{j}^{-\delta }$.

\item There exists an absolute constant $C_{2}<\infty $, such that for all $%
k\geq 2$, 
\begin{equation}
\mathbb{E}^{D_{N},0}\left[ E_{Y}^{\left( k\right) }1_{\mathcal{G}}\right]
\leq C_{2}\varepsilon ^{2}\mathbb{E}^{D_{N},0}\left[ Z_{k-2}1_{\mathcal{G}}%
\right] .  \label{ey}
\end{equation}
\end{itemize}

Notice that (\ref{Zgood}) implies (\ref{ind2}), since for all $k\geq 1$, 
\begin{equation}
\mathbb{E}^{r_{k-1},0}\left[ W_{k}1_{\mathcal{G}^{c}}\right] \leq \exp
\left( -\frac{c_{1}}{4}\left( \log r_{k-1}\right) ^{3}\right) ,
\label{Wkbad}
\end{equation}%
and similar bound hold for $\mathbb{E}^{D_{N},0}\left[ \exp \left(
tX_{r_{0,-}}\right) 1_{\mathcal{G}^{c}}\right] $. Clearly the base case $k=0$
for (\ref{Zgood}) is trivial.

Now assume both (\ref{Zgood}) and (\ref{ey}) hold up to $k-1$. Let us first
show the desired bound for $\mathbb{E}^{D_{N},0}\left[ E_{Y}^{\left(
k\right) }1_{\mathcal{G}}\right] $. For $m=1,...,k-2$, using the Markov
property and Cauchy-Schwarz, we conclude that each term in the first summand
of (\ref{eyk}) (multiplied by $1_{\mathcal{G}}$) can be bounded by 
\begin{eqnarray}
&&\mathbb{E}^{D_{N},0}\left[ W_{m+1}Z_{m}\prod_{j=m+2}^{k}\left( W_{j}\left(
Y_{j-1}-1\right) \right) 1_{\mathcal{G}}\right]  \notag \\
&=&\mathbb{E}^{D_{N},0}\left[ \mathbb{E}\left[ W_{m+1}\prod_{j=m+2}^{k}%
\left( W_{j}\left( Y_{j-1}-1\right) \right) 1_{\mathcal{G}}|\mathcal{F}_{m}%
\right] Z_{m}1_{\mathcal{G}}\right]  \notag \\
&\leq &\mathbb{E}^{D_{N},0}\left[ \mathbb{E}\left[
\prod_{j=m+1}^{k}W_{j}^{2}1_{\mathcal{G}}|\mathcal{F}_{m}\right] ^{1/2}%
\mathbb{E}\left[ \prod_{j=m+1}^{k-1}\left( Y_{j}-1\right) ^{2}1_{\mathcal{G}%
}|\mathcal{F}_{m}\right] ^{1/2}Z_{m}1_{\mathcal{G}}\right] .  \label{rest}
\end{eqnarray}%
We now claim that there exist constants $C_{3},C_{4}<\infty $, such that for 
$\left\vert t\right\vert \leq C$, 
\begin{equation}
\mathbb{E}^{r_{j-1,-},0}\left[ \left( Y_{j}-1\right) ^{2}\right] \leq
C_{3}\varepsilon ^{4},  \label{yj}
\end{equation}%
and 
\begin{equation}
\mathbb{E}^{r_{j-1},0}\left[ W_{j}^{2}\right] \leq \exp \left(
4c_{-}^{-1}t^{2}\text{Var}_{DGFF}^{r_{j-1},0}\left[ X_{r_{j,+}}-X_{r_{j-1,-}}%
\right] \right) \leq C_{4}.  \label{wj}
\end{equation}%
Indeed, using Taylor expansion\ we can write 
\begin{eqnarray*}
\mathbb{E}^{r_{j-1,-},0}\left[ \left( Y_{j}-1\right) ^{2}\right] &=&\mathbb{E}%
^{r_{j-1,-},0}\left[ \left( \sum_{k\geq 1}\frac{t^{k}}{k!}\left(
X_{r_{j,-}}-X_{r_{j,+}}\right) ^{k}\right) ^{2}\right] \\
&\leq &\sum_{k\geq 1}\mathbb{E}^{r_{j-1,-},0}\left[ \sum_{j=1}^{2k}\frac{2}{%
j!\left( 2k-j\right) !}t^{2k}\left( X_{r_{j,-}}-X_{r_{j,+}}\right) ^{2k}%
\right] .
\end{eqnarray*}%
Using the identity%
\begin{equation*}
\sum_{j=1}^{2k}\frac{1}{j!\left( 2k-j\right) !}=\frac{1}{\left( 2k\right) !}%
2^{2k},
\end{equation*}%
and the Brascamp-Lieb inequality (\ref{eq: BL even moment bound}) combined
with Wick's theorem, 
\begin{eqnarray*}
\mathbb{E}^{r_{j-1,-},0}\left[ \left( X_{r_{j,-}}-X_{r_{j,+}}\right) ^{2k}%
\right] &\leq &c_{-}^{-k}\mathbb{E}_{DGFF}^{r_{j-1,-},0}\left[ \left(
X_{r_{j,-}}-X_{r_{j,+}}\right) ^{2k}\right] \\
&\leq &c_{-}^{-k}\frac{\left( 2k\right) !}{k!2^{k}}\left( \mathbb{E}%
_{DGFF}^{r_{j-1,-},0}\left[ \left( X_{r_{j,-}}-X_{r_{j,+}}\right) ^{2}\right]
\right) ^{k},
\end{eqnarray*}%
we obtain 
\begin{eqnarray*}
\mathbb{E}^{r_{j-1,-},0}\left[ \left( Y_{j}-1\right) ^{2}\right] &\leq
&\sum_{k\geq 1}\frac{t^{2k}2^{k+1}}{k!}c_{-}^{-k}\left( \mathbb{E}%
_{DGFF}^{r_{j-1,-},0}\left[ \left( X_{r_{j,-}}-X_{r_{j,+}}\right) ^{2}\right]
\right) ^{k} \\
&\leq &C^{\prime }t^{2}\mathbb{E}_{DGFF}^{r_{j-1,-},0}\left[ \left(
X_{r_{j,-}}-X_{r_{j,+}}\right) ^{2}\right]
\end{eqnarray*}%
for some $C^{\prime }<\infty $. A similar computation as (\ref{gffdiff})
using the Gibbs-Markov property then yields 
\begin{equation}
\text{Var}_{DGFF}^{r_{j-1,-},0}\left[ X_{r_{j,-}}-X_{r_{j,+}}\right] \leq
C^{\prime \prime }\log \frac{r_{j,+}}{r_{j,-}}\leq C^{\prime \prime
}\varepsilon ^{4}.  \label{varyj}
\end{equation}%
This verifies (\ref{yj}). (\ref{wj}) follows directly from the exponential
Brascamp-Lieb inequality (\ref{eq: BL exp mom bound}).

We then use (\ref{yj}) and (\ref{wj}) to obtain an upper bound of (\ref{rest}%
). Let $\mathcal{F}_{k}^-=\sigma \left( \phi \left(
x\right) :x\in D_{N}\setminus B_{r_{k},-}\left( v\right) \right) $. Again use the Markov property%
\begin{eqnarray*}
&&\mathbb{E}\left[ \prod_{j=m+1}^{k-1}\left( Y_{j}-1\right) ^{2}1_{\mathcal{G%
}}|\mathcal{F}_{m}\right] \\
&=&\mathbb{E}\left[ \mathbb{E}\left[ \left( Y_{k-1}-1\right) ^{2}1_{\mathcal{%
G}}|\mathcal{F}_{k-2}^-\right] \prod_{j=m+1}^{k-2}\left( Y_{j}-1\right) ^{2}1_{%
\mathcal{G}}|\mathcal{F}_{m}\right] .
\end{eqnarray*}%
We now use the fact that $r_{k-2,-}\geq r_{M}\geq \Delta ^{c}$, and therefore
on the event $\mathcal{G}$, 
\begin{equation}
\max_{x\in \partial B_{r_{k-2},-}\left( v\right) }|\phi \left( x\right) |\leq
\left( \log \Delta \right) ^{2}\leq \left( \frac{1}{c}\log r_{k-2,-}\right)
^{2}.  \label{condk-2}
\end{equation}%
Apply Theorem \ref{decouple} (to any realization of $\phi |_{\partial
B_{r_{k-2,-}}}$ that satisfy (\ref{condk-2})), Cauchy-Schwarz and the
Brascamp-Lieb inequality, we conclude there is some $C_{3}^{\prime }<\infty $
and $\delta >0$, such that with probability one, 
\begin{equation*}
\left\vert \mathbb{E}\left[ \left( Y_{k-1}-1\right) ^{2}1_{\mathcal{G}}|%
\mathcal{F}_{k-2}^-\right] -\mathbb{E}^{r_{k-2,-},0}\left[ \left(
Y_{k-1}-1\right) ^{2}1_{\mathcal{G}}\right] \right\vert \leq C_{3}^{\prime
}\varepsilon ^{4}r_{k-2}^{-\delta },
\end{equation*}%
for some $\delta >0$. Thus%
\begin{eqnarray*}
&&\mathbb{E}\left[ \prod_{j=m+1}^{k-1}\left( Y_{j}-1\right) ^{2}1_{\mathcal{G%
}}|\mathcal{F}_{m}\right] \\
&=&\left( \mathbb{E}^{r_{k-2,-},0}\left[ \left( Y_{k-1}-1\right) ^{2}1_{%
\mathcal{G}}\right] +O\left( \varepsilon ^{4}r_{k-2}^{-\delta }\right)
\right) \mathbb{E}\left[ \prod_{j=m+1}^{k-2}\left( Y_{j}-1\right) ^{2}1_{%
\mathcal{G}}|\mathcal{F}_{m}\right] ,
\end{eqnarray*}%
here we also use the fact that $\prod_{j=m+1}^{k-2}\left( Y_{j}-1\right)
^{2}\geq 0$. By iterating this for $j\geq m+1$, applying the bound (\ref{yj}%
) and notice $\sum_j r_{j-2}^{-\delta }<\infty $, we conclude there exist
absolute constants $C_{2}^{\prime },C_{3}^{\prime \prime }<\infty $, such
that with probability one, 
\begin{equation*}
\mathbb{E}\left[ \prod_{j=m+1}^{k-1}\left( Y_{j}-1\right) ^{2}1_{\mathcal{G}%
}|\mathcal{F}_{m}\right] \leq C_{2}^{\prime }\left( C_{3}^{\prime \prime
}\varepsilon ^{4}\right) ^{k-m-1}.
\end{equation*}%
Similarly, there exists $C_{4}^{\prime }<\infty $, such that with probability one,
\begin{equation*}
\mathbb{E}\left[ \prod_{j=m+1}^{k}W_{j}^{2}1_{\mathcal{G}}|\mathcal{F}_{m}%
\right] \leq C_{2}^{\prime }\left( C_{4}^{\prime }\right) ^{k-m-1}.
\end{equation*}

Substitute these bounds into (\ref{rest}), we have for some $C_{5}<\infty $, 
\begin{equation}
\mathbb{E}^{D_{N},0}\left[ W_{m+1}Z_{m}\prod_{j=m+2}^{k}\left( W_{j}\left(
Y_{j-1}-1\right) \right) 1_{\mathcal{G}}\right] \leq C_{2}^{\prime }\left(
C_{5}\varepsilon ^{2}\right) ^{k-m-1}\mathbb{E}^{D_{N},0}\left[ Z_{m}1_{%
\mathcal{G}}\right] .  \label{rest1}
\end{equation}%
By the induction hypothesis (\ref{Zgood}) for $m\leq k-2$, 
\begin{equation*}
\log \mathbb{E}^{D_{N},0}\left[ Z_{m}1_{\mathcal{G}}\right] -\log \mathbb{E}%
^{D_{N},0}\left[ Z_{k-2}1_{\mathcal{G}}\right] \leq -\sum_{j=m+1}^{k-2}\log 
\mathbb{E}^{r_{j-1},0}\left[ W_{j}1_{\mathcal{G}}\right] +t^{2}\left\vert
F_{k-2}-F_{m}\right\vert +\left\vert R_{m-1}\right\vert ,
\end{equation*}%
where $\left\vert F_{k-2}-F_{m}\right\vert \leq C_{1}\left( k-2-m\right)
\varepsilon ^{2}=C_{1}\frac{\varepsilon ^{2}}{\log \left( 1+\varepsilon
\right) }\log \frac{r_{m}}{r_{k-2}}$, $\left\vert R_{m-1}\right\vert \leq
C_{1}r_{m-1}^{-\delta }$. Applying Proposition \ref{mgfCLT} (and use the
smallness of $\mathbb{E}^{r_{j-1},0}\left[ W_{j}1_{\mathcal{G}^{c}}\right] $%
) to evaluate $\log \mathbb{E}^{r_{j-1},0}\left[ W_{j}1_{\mathcal{G}}\right] 
$ as $\Delta \rightarrow \infty $, the right side is bounded above by%
\begin{equation*}
-\frac{t^{2}}{2}g\log \frac{r_{m}}{r_{k-2}}+t^{2}o_{\Delta }\left( 1\right)
\log \frac{r_{m}}{r_{k-2}}+O\left( \frac{\varepsilon ^{2}}{\log \left(
1+\varepsilon \right) }\right) \log \frac{r_{m}}{r_{k-2}}+O\left(
r_{m-1}^{-\delta }\right) .
\end{equation*}%
For $\varepsilon $ sufficiently small, this is bounded by $O\left(
r_{m-1}^{-\delta }\right) $, and we have $\mathbb{E}^{D_{N},0}\left[ Z_{m}1_{%
\mathcal{G}}\right] \leq 2\mathbb{E}^{D_{N},0}\left[ Z_{k-2}1_{\mathcal{G}}%
\right] $ for all $m\leq k-2$. This concludes that for some absolute
constant $C_{6}<\infty $, (\ref{rest1}) is bounded by 
\begin{equation*}
C_{6}\varepsilon ^{2(k-m-1)}\mathbb{E}^{D_{N},0}\left[ Z_{k-2}1_{\mathcal{G}}%
\right] .
\end{equation*}%
Summing over $m$, we then have%
\begin{equation*}
\mathbb{E}^{D_{N},0}\left[ \sum_{m=1}^{k-2}W_{m+1}Z_{m}\prod_{j=m+2}^{k}%
\left( W_{j}\left( Y_{j-1}-1\right) \right) 1_{\mathcal{G}}\right] \leq
C_{7}\varepsilon ^{2}\mathbb{E}^{D_{N},0}\left[ Z_{k-2}1_{\mathcal{G}}\right]
,
\end{equation*}%
for some $C_{7}<\infty $. A similar argument yields%
\begin{equation*}
\mathbb{E}^{D_{N},0}\left[ Z_{1}\prod_{j=2}^{k}W_{j}\left( Y_{j-1}-1\right)
1_{\mathcal{G}}\right] \leq C_{7}^{\prime }\varepsilon ^{2}\mathbb{E}%
^{D_{N},0}\left[ Z_{k-2}1_{\mathcal{G}}\right] .
\end{equation*}%
This finishes the proof of (\ref{ey}) for $k$ (with $C_{2}=C_{7}+C_{7}^{%
\prime }$).

We now move to the proof of (\ref{Zgood}). We first show that there exists $%
C_{0}<\infty $, such that%
\begin{equation}
\log \mathbb{E}^{D_{N},0}\left[ W_{k}Z_{k-1}1_{\mathcal{G}}\right] =\log 
\mathbb{E}^{D_{N},0}\left[ Z_{k-1}1_{\mathcal{G}}\right] +\log \mathbb{E}%
^{r_{k-1},0}\left[ W_{k}1_{\mathcal{G}}\right] +R_{k-1}^{\prime },
\label{wz}
\end{equation}%
where $\left\vert R_{k-1}^{\prime }\right\vert \leq C_{0}r_{k-1}^{-\delta }$%
. Using Markov property, 
\begin{equation*}
\mathbb{E}^{D_{N},0}\left[ W_{k}Z_{k-1}1_{\mathcal{G}}\right] =\mathbb{E}%
^{D_{N},0}\left[ Z_{k-1}1_{\mathcal{G}}\mathbb{E}\left[ W_{k}1_{\mathcal{G}}|%
\mathcal{F}_{k-1}\right] \right]
\end{equation*}

Apply Lemma \ref{jasonerr} (to any realization of $\phi |_{\partial
B_{r_{k-1}}}$ that is absoluted bounded by $\left( \frac{1}{c}\log
r_{k-1}\right) ^{2}$) to obtain with probability one, there is some $%
C_{0}<\infty $ and $\delta >0$, such that 
\begin{eqnarray*}
|\mathbb{E}\left[ W_{k}1_{\mathcal{G}}|\mathcal{F}_{k-1}\right] -\mathbb{E}%
^{r_{k-1},0}\left[ W_{k}1_{\mathcal{G}}\right] | &\leq &2\exp \left( c\text{%
Var}_{DGFF}^{r_{k-1},0}W_{k}\right) r_{k-1}^{-\delta } \\
&\leq &C_{0}r_{k-1}^{-\delta },
\end{eqnarray*}
Therefore, 
\begin{equation*}
\left\vert \mathbb{E}^{D_{N},0}\left[ W_{k}Z_{k-1}1_{\mathcal{G}}\right] -%
\mathbb{E}^{D_{N},0}\left[ Z_{k-1}1_{\mathcal{G}}\right] \mathbb{E}%
^{r_{k-1},0}\left[ W_{k}1_{\mathcal{G}}\right] \right\vert \leq
C_{0}r_{k-1}^{-\delta }\mathbb{E}^{D_{N},0}\left[ Z_{k-1}1_{\mathcal{G}}%
\right] .
\end{equation*}%
This yields (\ref{wz}).

\bigskip Finally, we prove (\ref{Zgood}) for $k$ using the joint induction
hypothesis for (\ref{Zgood}) up to $k-1$ and for (\ref{ey}) up to $k$, and
apply (\ref{wz}). By (\ref{wz}), and the induction hypothesis for (\ref%
{Zgood}), 
\begin{eqnarray}
&&\log \mathbb{E}^{D_{N},0}\left[ W_{k}Z_{k-1}1_{\mathcal{G}}\right]  \notag
\\
&=&\log \mathbb{E}^{D_{N},0}\left[ Z_{k-1}1_{\mathcal{G}}\right] +\log 
\mathbb{E}^{r_{k-1},0}\left[ W_{k}1_{\mathcal{G}}\right] +R_{k-1}^{\prime } 
\notag \\
&=&\sum_{j=1}^{k}\log \mathbb{E}^{r_{j-1},0}\left[ W_{j}1_{\mathcal{G}}%
\right] +\log \mathbb{E}^{D_{N},0}\left[ \exp \left( t\left(
X_{r_{0,-}}\right) \right) 1_{\mathcal{G}}\right]
+t^{2}F_{k-1}+R_{k-2}+R_{k-1}^{\prime }.  \label{wziterate}
\end{eqnarray}%
We may write%
\begin{equation*}
\log \mathbb{E}^{D_{N},0}\left[ Z_{k}1_{\mathcal{G}}\right] =\log \mathbb{E}%
^{D_{N},0}\left[ W_{k}Z_{k-1}1_{\mathcal{G}}\right] +\log \left[ 1+\frac{%
\mathbb{E}^{D_{N},0}\left[ E_{Y}^{\left( k\right) }1_{\mathcal{G}}\right] }{%
\mathbb{E}^{D_{N},0}\left[ W_{k}Z_{k-1}1_{\mathcal{G}}\right] }\right]
\end{equation*}%
Using the induction hypothesis for (\ref{ey}), (\ref{wziterate}), and the
asymptotics of $\log \mathbb{E}^{r_{j-1},0}\left[ W_{j}\right] $, we
conclude for some absolute constant $C_{8}<\infty $, 
\begin{equation*}
\left\vert \log \left[ 1+\frac{\mathbb{E}^{D_{N},0}\left[ E_{Y}^{\left(
k\right) }1_{\mathcal{G}}\right] }{\mathbb{E}^{D_{N},0}\left[ W_{k}Z_{k-1}1_{%
\mathcal{G}}\right] }\right] \right\vert \leq C_{8}\varepsilon ^{2}\frac{%
\mathbb{E}^{D_{N},0}\left[ Z_{k-2}1_{\mathcal{G}}\right] }{\mathbb{E}%
^{D_{N},0}\left[ W_{k}Z_{k-1}1_{\mathcal{G}}\right] }\leq 2C_{8}\varepsilon
^{2}.
\end{equation*}%
Let $F_{k}=F_{k-1}+\log \left[ 1+\frac{\mathbb{E}^{D_{N},0}\left[
E_{Y}^{\left( k\right) }1_{\mathcal{G}}\right] }{\mathbb{E}^{D_{N},0}\left[
W_{k}Z_{k-1}1_{\mathcal{G}}\right] }\right] $ and $R_{k-1}=R_{k-2}+R_{k-1}^{%
\prime }.$ Combining with (\ref{wziterate}) we conclude 
\begin{equation*}
\log \mathbb{E}^{D_{N},0}\left[ Z_{k}1_{\mathcal{G}}\right]
=\sum_{j=1}^{k}\log \mathbb{E}^{r_{j-1},0}\left[ W_{j}1_{\mathcal{G}}\right]
+\log \mathbb{E}^{D_{N},0}\left[ \exp \left( t\left( X_{r_{0,-}}\right)
\right) 1_{\mathcal{G}}\right] +t^{2}F_{k}+R_{k-1},
\end{equation*}%
with $\left\vert F_{k}\right\vert \leq \max \left\{ C_{1},2C_{8}\right\}
k\varepsilon ^{2}$, $\left\vert R_{k-1}\right\vert \leq
C_{1}\sum_{j=1}^{k-1}r_{j}^{-\delta }$. This finishes the proof of (\ref%
{Zgood}), and also Theorem \ref{ind}.

\subsection{Proof of upper bound}

\label{sec: ub}

In this section we prove the pointwise Gaussian tail bound Theorem \ref{thm:
tail bound}, and as a consequence derive the upper bound of the law of large
numbers Theorem \ref{main}. In fact, we obtain the following tail bound for
the maximum of $\phi (x)$. For the rest of the paper, $g=g\left( V\right) $
denotes the the positive constant that appears in Theorem \ref{mgfCLT} and
Theorem \ref{Gauss}.

\begin{proposition}
\label{prop: upper tail bound}For any $\delta >0$, there is some $C=C\left(
\delta \right) <\infty $, such that 
\begin{equation}
\mathbb{P}\left( \sup_{v\in D_{N}}\phi \left( v\right) \geq \left( 2\sqrt{g}%
+\delta \right) \log N\right) \leq C\left( \delta \right) N^{-\delta /\sqrt{g%
}}.  \label{ub}
\end{equation}
\end{proposition}

We first give the proof of Theorem \ref{thm: tail bound}.

\begin{proof}[Proof of Theorem \protect\ref{thm: tail bound}]
Given $\delta >0$ and $v\in D_{N}$, take $M=M\left( \delta \right) =\left(
1-\delta ^{6}\right) \log \Delta $. Therefore%
\begin{equation*}
\mathbb{P}^{D_{N},0}\left( \phi \left( v\right) >u\right) \leq \mathbb{P}%
^{D_{N},0}\left( X_{r_{M,+}}\left( v\right) >u-\delta \log \Delta \right) +%
\mathbb{P}^{D_{N},0}\left( \phi \left( v\right) -X_{r_{M,+}}\left( v\right)
>\delta \log \Delta \right) .
\end{equation*}%
We apply Theorem \ref{Gauss} to obtain for all bounded $t$,%
\begin{eqnarray*}
\mathbb{P}^{D_{N},0}\left( X_{r_{M,+}}>u-\delta \log \Delta \right) &\leq
&\exp \left( -t\left( u-\delta \log \Delta \right) \right) \mathbb{E}%
^{D_{N},0}\left[ \exp \left( tX_{r_{M,+}}\right) \right] \\
&=&\exp \left( -t\left( u-\delta \log \Delta \right) +\frac{t^{2}}{2}g\left(
1-\delta ^{6}\right) \log \Delta +o\left( \log \Delta \right) \right) .
\end{eqnarray*}%
Minimize the last display over $t$. Since $u\leq C\log \Delta $ the minimum
is achieved at some bounded $t$, thus 
\begin{eqnarray*}
\mathbb{P}^{D_{N},0}\left( X_{r_{M,+}}>u-\delta \log \Delta \right) &\leq
&\exp \left( -\frac{\left( u-\delta \log \Delta \right) ^{2}}{2g\left(
1-\delta ^{6}\right) \log \Delta }+o\left( \log \Delta \right) \right) \\
&\leq &\exp \left( -\frac{\left( u-\delta \log \Delta \right) ^{2}}{2g\log
\Delta }+o\left( \log \Delta \right) \right) .
\end{eqnarray*}%
Apply Lemma \ref{lem: BL tail bound} to obtain%
\begin{equation*}
\mathbb{P}^{D_{N},0}\left( \phi \left( v\right) -X_{r_{M,+}}>\delta \log
\Delta \right) \leq \exp \left( -c_{BL}\frac{\left( \delta \log \Delta
\right) ^{2}}{g\delta ^{6}\log \Delta }\right) =\exp \left( -c_{BL}\frac{%
\log \Delta }{g\delta ^{4}}\right) .
\end{equation*}%
Notice that for $\delta $ small enough, 
\begin{equation*}
2c_{BL}\frac{\log \Delta }{g\delta ^{4}}>\frac{\left( u-\delta \log \Delta
\right) ^{2}}{2g\log \Delta },
\end{equation*}%
so we send $\delta \rightarrow 0$ to conclude the proof.
\end{proof}

Finally we show how Proposition \ref{prop: upper tail bound} follows easily
from Theorem \ref{thm: tail bound}.

\begin{proof}[Proof of Proposition \protect\ref{prop: upper tail bound}]
If we pick $\gamma _{0}$ small enough then for $v\in D_{N}$ such that dist$%
(v,\partial D_{N})\leq N^{\gamma _{0}}$ we have from the Brascamp-Lieb tail
bound, Lemma \ref{lem: BL tail bound}, that 
\begin{equation*}
P\left( \phi (v)\geq 2\sqrt{g}\log N\right) \leq \exp \left( -c_{BL}\frac{%
4g\left( \log N\right) ^{2}}{\gamma _{0}\log N}\right) \leq N^{-2-2\delta /%
\sqrt{g}}.
\end{equation*}%
Then a union bound shows that 
\begin{equation*}
P\left( \max_{v:\text{dist}(v,\partial D_{N})\leq N^{\gamma _{0}}}\phi
(v)\geq 2\sqrt{g}\log N\right) \leq N^{\gamma _{0}-1-2\delta /\sqrt{g}}.
\end{equation*}%
Fix this $\gamma _{0}$ and take any $v\in D_{N}$ such that dist$(v,\partial
D_{N})>N^{\gamma _{0}}$. Given any $\delta >0$, applying Proposition \ref%
{thm: tail bound} with $u=\left( 2\sqrt{g}+\delta \right) \log N$ yields 
\begin{eqnarray*}
P\left( \phi (v)\geq \left( 2\sqrt{g}+\delta \right) \log N\right) &\leq
&\exp \left( -2\frac{\left( \log N\right) ^{2}}{\log \Delta }-\frac{2\delta 
}{\sqrt{g}}\frac{\left( \log N\right) ^{2}}{\log \Delta }+o\left( \log
N\right) \right) \\
&\leq &CN^{-2-2\delta /\sqrt{g}+o\left( 1\right) },
\end{eqnarray*}%
for some $C<\infty $. Therefore 
\begin{equation*}
P\left( \max_{v:\text{dist}(v,\partial D_{N})>N^{\gamma _{0}}}\phi (v)\geq 2%
\sqrt{g}\log N\right) \leq CN^{-2\delta /\sqrt{g}+o\left( 1\right) },
\end{equation*}%
thus completing the proof of (\ref{ub}).
\end{proof}

\section{Proof of the Lower Bound\label{LB}}

In this section we prove the lower bound of the law of large numbers Theorem %
\ref{main}. In fact we prove the following tail bound.

\begin{proposition}
For any $\delta >0$, there is some $C=C\left( \delta \right) <\infty $, such
that 
\begin{equation}
\mathbb{P}^{D_{N},0}\left( \sup_{v\in D_{N}}\phi \left( v\right) \leq \left(
2\sqrt{g}-\delta \right) \log N\right) \leq C\left( \delta \right)
N^{-C\delta ^{-1}}.  \label{lb}
\end{equation}
\end{proposition}

We first prove a weaker form of the lower bound in Section \ref{2nd}, and
then \textquotedblleft bootstrap\textquotedblright\ to obtain the desired
lower bound in Section \ref{boot}. Recall that $\mathbb{P}^{B,f}$ represents
the law of the gradient field in $B\subset \mathbb{Z}^{2}$ with boundary
condition $f$ on $\partial B$.

\subsection{Second moment argument\label{2nd}}

Given $B\subset \mathbb{Z}^{2}$, $x\in B$ and $y\in \partial B$, we recall $%
a_{B}\left( x,y\right) $ is the harmonic measure on $\partial B$ seen from $%
x $. Also recall the harmonic averaged field $X_{r_{j},+}\left( v\right) $
and $X_{r_{j},-}\left( v\right) $ from the beginning of Section \ref{ptwise}%
. Heuristically, the process $\left\{ X_{r_{j},+}\left( v\right) \right\} $
should behave like a random walk with increments of variance $g\log \left(
1+\varepsilon \right) $. We make this heuristic rigorous and show the
following weak lower bound:

\begin{proposition}
\label{prop: weak bound} For all $s>0$, there is $N_{0}=N_{0}\left( s\right) 
$ such that for $N>N_{0}\left( s\right) $%
\begin{equation}
\mathbb{P}^{D_{N},0}\left[ 
\begin{array}{c}
\exists v\in \left[ -0.9N,0.9N\right] ^{2}\mbox{ s.t. } \\ 
\phi (v)-X_{r_{_{0},-}}(v)\geq (1-2s)2\sqrt{g}\log N%
\end{array}%
\right] \geq N^{-22s}.  \label{eq: weak bound}
\end{equation}
\end{proposition}

\bigskip In fact, this probability tends to one as $N\rightarrow \infty $.
This will be proved later by bootstrapp{ing} the weaker bound stated in
Proposition \ref{prop: weak bound}. The proof of Proposition \ref{prop: weak
bound} is based on a second moment method studying the truncated count of
the increment of the harmonic averaged process.

It suffices to prove Proposition \ref{prop: weak bound} for small $s$. Given 
$v\in \left[ -0.9N,0.9N\right] ^{2}$, take $c=s^{3}$ and $M=M\left(
s^{3}\right) =\left( 1-s^{3}\right) \log N/\log \left( 1+\varepsilon \right) 
$, and define $r_{k}$ and $r_{k,\pm }$ as in (\ref{scale}). Denote by $\left[
m\right] $ the integer part of $m$. Then we have 
\begin{eqnarray*}
&&\mathbb{P}^{D_{N},0}\left[ 
\begin{array}{c}
\exists v\in \left[ -0.9N,0.9N\right] ^{2}\mbox{ s.t. } \\ 
\phi (v)-X_{r_{_{0},-}}(v)\geq (1-2s)2\sqrt{g}\log N%
\end{array}%
\right]  \\
&\geq &\mathbb{P}^{D_{N},0}\left[ 
\begin{array}{c}
\exists v\in \left[ -0.9N,0.9N\right] ^{2}\mbox{ s.t. } \\ 
X_{r_{\left[ M\right] },+}\left( v\right) -X_{r_{_{0},-}}(v)\geq (1-\frac{3}{%
2}s)2\sqrt{g}\log N%
\end{array}%
\right]  \\
&&-\mathbb{P}^{D_{N},0}\left[ 
\begin{array}{c}
\exists v\in \left[ -0.9N,0.9N\right] ^{2}\mbox{ s.t. } \\ 
\phi \left( v\right) -X_{r_{\left[ M\right] },+}(v)\leq - \frac{s}{2}2\sqrt{g}%
\log N%
\end{array}%
\right] .
\end{eqnarray*}%
The last term above can be bounded using the Brascamp-Lieb inequality.
Indeed, 
\begin{eqnarray}
&&\mathbb{P}^{D_{N},0}\left[ 
\begin{array}{c}
\exists v\in \left[ -0.9N,0.9N\right] ^{2}\mbox{ s.t. } \\ 
\phi \left( v\right) -X_{r_{\left[ M\right] },+}(v)\leq -\frac{s}{2}2\sqrt{g}%
\log N%
\end{array}%
\right]   \notag \\
&\leq &\sum_{v\in \left[ -0.9N,0.9N\right] ^{2}}\mathbb{P}^{D_{N},0}\left[
\phi \left( v\right) -X_{r_{\left[ M\right] },+}(v)\geq \frac{s}{2}2\sqrt{g}%
\log N\right]   \notag \\
&\leq &N^{2}\exp \left( -c_{BL}\frac{s^{2}g\left( \log N\right) ^{2}}{\text{%
Var}_{DGFF}^{D_{N},0}\left( \phi \left( v\right) -X_{r_{\left[ M\right]
},+}(v)\right) }\right) \notag \\
&\leq &N^{2}\exp \left( -c_{BL}\frac{s^{2}g\left( \log N\right) ^{2}}{%
gs^{3}\log N}\right) =N^{2-c^{\prime }s^{-1}},  \label{last}
\end{eqnarray}%
for some $c^{\prime }>0$. For small $s$ this is much smaller than $N^{-22s}$%
. Therefore it suffices to study $X_{r_{\left[ M\right]
},+}(v)-X_{r_{_{0},-}}(v)$.

For fixed integer $K\geq 2$ (which will be taken sufficiently large in the
end), split \textquotedblleft time\textquotedblright\ into $K_{1}:=\left[
\left( 1-s^{3}\right) K\right] +1$ intervals of size $1/K$ and consider the
increments over these intervals 
\begin{equation}
\begin{array}{lll}
U_{m}(v) & = & X_{r_{\left[ \frac{mM}{K_{1}}\right] ,+}}(v)-X_{r_{\left[ 
\frac{\left( m-1\right) M}{K_{1}}\right] ,-}}(v)%
\end{array}%
,\text{ for }m=1,...,K_{1}.  \label{Um}
\end{equation}%
Roughly speaking, when $v$ is in the bulk of $D_{N}$, $\left\{ U_{m}\right\}
_{m=1}^{K_{1}}$ are the differences between the harmonic average at scale $%
N^{1-m/K}$ and the scale $N^{1-\left( m-1\right) /K}$. Consider the events 
\begin{equation*}
J_{m}(v;s)=\left\{ U_{m}\left( v\right) \in \left[ \frac{1}{K}(1-s)2\sqrt{g}%
\log N,\frac{1}{K}(1+s)2\sqrt{g}\log N\right] \right\} .
\end{equation*}%
and 
\begin{equation*}
J(v;s)=\bigcap_{m=1,\ldots ,K_{1}}J_{m}(v;s).
\end{equation*}%
Define the counting random variable 
\begin{equation*}
\mathcal{N}_{K_{1}}(s)=\sum_{v\in \left[ -0.9N,0.9N\right] ^{2}}1_{J(v;s)}.
\end{equation*}%
Note that if $\mathcal{N}_{K_{1}}(s)\geq 1$ then there exists a $v\in \left[
-0.9N,0.9N\right] ^{2}$ such that%
\begin{equation*}
\sum_{m=1}^{K_{1}}U_{m}\left( v\right) \geq (1-s)\left( 1-s^{3}\right) 2%
\sqrt{g}\log N\geq \left( 1-\frac{5}{4}s\right) 2\sqrt{g}\log N.
\end{equation*}%
Furthermore, since%
\begin{equation*}
X_{r_{\left[ M\right] },+}(v)-X_{r_{_{0},-}}(v)=\sum_{m=1}^{K_{1}}U_{m}%
\left( v\right) +\sum_{m=1}^{K_{1}}\left( X_{\left[ mM/K_{1}\right]
,-}\left( v\right) -X_{\left[ mM/K_{1}\right] ,+}\left( v\right) \right) ,
\end{equation*}%
and by direct computation 
\begin{equation*}
\text{Var}_{DGFF}^{D_{N},0}\left[ \sum_{m=1}^{K_{1}}\left( X_{\left[ mM/K_{1}%
\right] ,-}\left( v\right) -X_{\left[ mM/K_{1}\right] ,+}\left( v\right)
\right) \right] =O\left( K_{1}\right) ,
\end{equation*}%
the Brascamp-Lieb tail bound Lemma \ref{lem: BL tail bound} implies there
exists some $c\left( s,K_{1}\right) >0$, such that 
\begin{equation}
\mathbb{P}^{D_{N},0}\left( \sum_{m=1}^{K_{1}}\left( X_{\left[ mM/K_{1}\right]
,-}\left( v\right) -X_{\left[ mM/K_{1}\right] ,+}\left( v\right) \right) >2%
\sqrt{g}\frac{s}{4}\log N\right) \leq e^{-c\left( s,K_{1}\right) \left( \log
N\right) ^{2}}.  \label{xerr}
\end{equation}%
Combining (\ref{last}) and (\ref{xerr}), Proposition \ref{prop: weak bound}
will follow from 
\begin{equation}
\mathbb{P}^{D_{N},0}\left[ \mathcal{N}_{K_{1}}\left( s\right) \geq 1\right]
\geq N^{-22s}.  \label{pz}
\end{equation}%
We will prove the following.

\begin{lemma}
\label{lem: fm sm} For all $s>0$ and $K\geq 2/s$ we have 
\begin{equation}
\mathbb{E}^{D_{N},0}[\mathcal{N}_{K_{1}}(s)^{2}]\leq N^{22s}\mathbb{E}%
^{D_{N},0}[\mathcal{N}_{K_{1}}(s)]^{2}.  \label{eq: fm sm}
\end{equation}
\end{lemma}

With additional work, the term $N^{22s}$ could be replaced $(1+o(1))$, but
for our purposes \eqref{eq: fm sm} is enough. Note that \eqref{eq: fm sm} is
true only because $\mathcal{N}_{K_{1}}(s)$ is a \textit{truncated} count of
high points.

By the Paley-Zygmund inequality, Lemma \ref{lem: fm sm} implies (\ref{pz}),
and therefore yields Proposition \ref{prop: weak bound}.

Lemma \ref{lem: fm sm} follows from the following estimates:

\begin{lemma}
\label{lem: fm} For all fixed $s>0$ and $K\geq 2$ we have 
\begin{equation*}
\mathbb{E}^{D_{N},0}[\mathcal{N}_{K_{1}}(s)]\geq cN^{-5s}.
\end{equation*}
\end{lemma}

\begin{lemma}
\label{lem: sm} For all fixed $s>0$ and $K\geq 2$ we have 
\begin{equation*}
\mathbb{E}^{D_{N},0}[\mathcal{N}_{K_{1}}(s)^{2}]\leq N^{\frac{2}{K}+11s}.
\end{equation*}%
%
%
%
%
%
%
%
%
%
%
%
%
%
%
%
%
%
%
%
%
%
%
%
%
%
%
%
%
%
%
\end{lemma}

The proof of these lemmas use estimates for the joint distribution of $%
\left\{ U_{m}\right\} _{m=1}^{K_{1}}$, proved in Section \ref{fdd} below.
Lemma \ref{lem: fm} is immediate from taking union bound from the following
result.

\begin{lemma}
\label{1pt}For all fixed $s >0$ and $K\geq 2$ we have that 
\begin{equation*}
\mathbb{P}^{D_{N},0}[J(v;s )]\geq cN^{-2-5s },
\end{equation*}%
uniformly over $v\in \left[ -0.9N,0.9N\right] ^{2}$.
\end{lemma}

\begin{proof}
Letting $\frac{dQ}{d\mathbb{P}^{D_{N},0}}=\frac{\exp (\lambda
\sum_{m=1}^{K_{1}}U_{m}(v))}{\mathbb{E}^{D_{N},0}[\exp (\lambda
\sum_{m=1}^{K_{1}}U_{m}(v))]}$ we have 
\begin{equation*}
\begin{array}{lll}
\mathbb{P}^{D_{N},0}[J(v,s)] & = & Q[J(v,s);e^{-\lambda
\sum_{m=1}^{K_{1}}U_{m}(v)}]\mathbb{E}^{D_{N},0}[\exp (\lambda
\sum_{m=1}^{K_{1}}U_{m}(v))] \\ 
& \geq & Q[J(v,s)]e^{-\lambda (1+s)\left( 1-s^{3}\right) 2\sqrt{g}\log N}%
\mathbb{E}^{D_{N},0}[\exp (\lambda \sum_{m=1}^{K_{1}}U_{m}(v))]%
\end{array}%
\end{equation*}%
By Theorem \ref{expcircle}, for all $\lambda \leq 2/\sqrt{g}$, 
\begin{equation}
\mathbb{E}^{D_{N},0}[\exp (\sum_{m=1}^{K_{1}}\lambda U_{m}(v))]=\exp (\frac{1%
}{2}\sum_{m=1}^{K_{1}}\lambda ^{2}\frac{1}{K}g\log N+o\left( \log N\right) ),
\label{eq: one point exp moment}
\end{equation}%
Therefore 
\begin{equation*}
\mathbb{P}^{D_{N},0}[J(v,s)]\geq Q[J(v,s)]e^{\frac{1}{2}\lambda ^{2}\left(
1-s^{3}\right) g\log N-\lambda (1+s)\left( 1-s^{3}\right) 2\sqrt{g}\log
N+o\left( \log N\right) }.
\end{equation*}%
Setting $\lambda =2/\sqrt{g}$ we find that 
\begin{equation*}
\mathbb{P}^{D_{N},0}[J(v,s)]\geq Q[J(v,s)]e^{-2\log N-5s\log N}.
\end{equation*}%
It thus only remains to show that $Q[J(v)]\geq c$. Under $Q$ we have for
each $j$ that 
\begin{eqnarray}
&&Q[\exp (t(U_{j}(v)-\frac{1}{K}2\sqrt{g}\log N))]  \label{eq: Q exp mom} \\
&=&\frac{\mathbb{E}^{D_{N},0}[\exp (\sum_{m=1}^{K_{1}}(\lambda
+1_{\{m=j\}}t)U_{m}(v)]}{\mathbb{E}^{D_{N},0}[\exp
(\sum_{m=1}^{K_{1}}\lambda U_{m}(v)]}\exp (-2t\frac{1}{K}\sqrt{g}\log N). 
\notag
\end{eqnarray}%
Thus applying Theorem \ref{expcircle} (with $\max \lambda _{i}=2/\sqrt{g}+1$%
) we have that \eqref{eq: Q exp
mom} equals 
\begin{equation*}
\begin{array}{l}
\frac{\exp \left( \frac{1}{2}\lambda ^{2}\left( 1-s^{3}\right) g\log
N+\lambda t\frac{1}{K}g\log N+\frac{1}{2}t^{2}\frac{1}{K}g\log N+o\left(
\log N\right) \right) }{\exp \left( \frac{1}{2}\lambda ^{2}\left(
1-s^{3}\right) g\log N+o\left( \log N\right) \right) }\exp (-2t\frac{1}{K}%
\sqrt{g}\log N) \\ 
=\exp \left( \frac{1}{2}t^{2}\frac{1}{K}g\log N+o\left( \log N\right)
\right) ,%
\end{array}%
\end{equation*}%
where the last equality follows because $\lambda =2/\sqrt{g}$. Using the
exponential Chebyshev inequality with $t=\pm s/\sqrt{g}$ therefore shows
that 
\begin{equation*}
Q[|U_{j}-\frac{1}{K}\sqrt{g}\log N|\geq s\frac{1}{K}\sqrt{g}\log N]\leq \exp
(-c\frac{s^{2}}{K}\log N),
\end{equation*}%
for some $c>0$. Thus $Q[J(v,s)]\geq 1-K\exp (-c\frac{s^{2}}{K}\log
N)\rightarrow 1$, as $N\rightarrow \infty $ for all $K$ and $s$.
\end{proof}

Lemma \ref{lem: sm} will follow from the following.

\begin{lemma}
\label{2pt} For all fixed $s>0$ and $K\geq 1$ we have if $N^{1-\frac{j}{K}%
}\leq |v_{1}-v_{2}|\leq N^{1-\frac{j-1}{K}}$ for some $j\in \left\{ 1,\ldots
,K_{1}\right\} ,$ then 
\begin{equation*}
\mathbb{P}^{D_{N},0}[J(v_{1},s)\cap J(v_{2},s)]\leq \exp (-2\frac{2K_{1}-j}{K%
}\log N+5s\frac{2K_{1}-j}{K}\log N).
\end{equation*}
\end{lemma}

\begin{proof}
Note that $B_{N^{1-\frac{j}{K}}}(v_{i})$ for $i=1,2$ are disjoint, but $%
B_{N^{1-\frac{j-1}{K}}}(v_{i})$ are not. Thus, roughly speaking, the
increments $U_{j+1}(v_{i})$ for $i=1,2$ depend on disjoint regions but $%
U_{j}(v_{i})$ do not. Because of this we expect $U_{m}(v_{i}),i=1,2$ to be
correlated for $m=1,\ldots ,j$ (and essentially perfectly correlated if $%
m\leq j-1$), but essentially independent for $m=j+1,\ldots ,K_{1}$. With
this in mind we in fact bound 
\begin{equation*}
\mathbb{P}^{D_{N},0}[J^{\prime }],
\end{equation*}%
where 
\begin{equation*}
J^{\prime }=\cap _{m=1,\ldots ,K_1}J_{m}(v_{1},s)\bigcap \cap _{m=j+1,\ldots
,K_{1}}J_{m}(v_{2},s),
\end{equation*}%
i.e., we drop the condition on $v_{2}$ for $m=1,\ldots ,j$.

Letting $\frac{dQ}{d\mathbb{P}^{D_{N},0}}=\frac{\exp
(\sum_{m=1}^{K_{1}}\lambda U_{m}(v_{1})+\lambda
\sum_{m=j+1}^{K_{1}}U_{m}(v_{2})))}{\mathbb{E}^{D_{N},0}[\exp
(\sum_{m=1}^{K_{1}}\lambda U_{m}(v_{1})+\lambda
\sum_{m=j+1}^{K_{1}}U_{m}(v_{2}))]}$ we have 
\begin{equation*}
\begin{array}{lll}
&  & \mathbb{P}^{D_{N},0}[J^{\prime }] \\ 
& \leq & 
\begin{array}{l}
Q[J^{\prime };\exp \left( -\sum_{m=1}^{K_{1}}\lambda U_{m}(v_{1})-\lambda
\sum_{m=j+1}^{K_{1}}U_{m}(v_{2}))\right) ] \\ 
\mathbb{E}^{D_{N},0}[\exp (\sum_{m=1}^{K_{1}}\lambda U_{m}(v_{1})+\lambda
\sum_{m=j+1}^{K_{1}}U_{m}(v_{2})))]%
\end{array}
\\ 
& \leq & 
\begin{array}{l}
\exp \left( -\lambda \frac{2K_{1}-j}{K}\left( 1-s\right) 2\sqrt{g}\log
N\right) \\ 
\mathbb{E}^{D_{N},0}[\exp (\sum_{m=1}^{K_{1}}\lambda U_{m}(v_{1})+\lambda
\sum_{m=j+1}^{K_{1}}U_{m}(v_{2})))]%
\end{array}%
\end{array}%
\end{equation*}%
By Theorem \ref{expcircle}, for all $\lambda \leq 2/\sqrt{g}$, 
\begin{equation}
\begin{array}{l}
\mathbb{E}^{D_{N},0}[\exp (\sum_{m=1}^{K_{1}}\lambda
U_{m}(v_{1})+\sum_{m=j+1}^{K_{1}}\lambda U_{m}(v_{2})))] \\ 
=\exp (\frac{1}{2}\sum_{m=1}^{K_{1}}\lambda ^{2}\frac{1}{K}g\log N+\frac{1}{2%
}\sum_{m=j+1}^{K_{1}}\lambda ^{2}\frac{1}{K}g\log N+o\left( \log N\right) ).%
\end{array}
\label{eq: two point exp moment}
\end{equation}%
Thus in fact $\mathbb{P}^{D_{N},0}[J^{\prime }]$ is at most 
\begin{equation*}
\exp \left( \frac{1}{2}\lambda ^{2}\frac{2K_{1}-j}{K}g\log N-\lambda \frac{%
2K_{1}-j}{K}(1-s)2\sqrt{g}\log N+o\left( \log N\right) \right) .
\end{equation*}%
Setting $\lambda =2/\sqrt{g}$ we find that 
\begin{equation*}
\mathbb{P}^{D_{N},0}[J^{\prime }]\leq \exp \left( -2\frac{2K_{1}-j}{K}\log
N+5s\frac{2K_{1}-j}{K}\log N\right) .  \label{J'}
\end{equation*}
\end{proof}

We can now prove the second moment estimate Lemma \ref{lem: sm}.

\begin{proof}[Proof of Lemma \protect\ref{lem: sm}]
We write the second moment as 
\begin{equation*}
\mathbb{E}^{D_{N},0}[\mathcal{N}_{K_{1}}^{2}]\leq \sum_{v_{1},v_{2}\in \left[
-0.9N,0.9N\right] ^{2}}\mathbb{P}^{D_{N},0}[J(v_{1},s)\cap J(v_{2},s)].
\end{equation*}%
Splitting the sum according to the distance $|v_{1}-v_{2}|$ we get that, 
\begin{multline*}
\mathbb{E}^{D_{N},0}[\mathcal{N}_{K_{1}}^{2}]=\sum_{j=1}^{K_{1}}%
\sum_{N^{1-j/K}\leq |v_{1}-v_{2}|\leq N^{1-(j-1)/K}}\mathbb{P}%
^{D_{N},0}[J(v_{1},s)\cap J(v_{2},s)] \\ +\sum_{\left\vert
v_{1}-v_{2}\right\vert \leq N^{s^{3}}}\mathbb{P}^{D_{N},0}[J(v_{1},s)\cap
J(v_{2},s)].
\end{multline*}%
The first summation gives the main contribution. Now using Lemma \ref{2pt}
and the fact that there are at most $N^{2}\times N^{2-2(j-1)/K}$ points at
distance less than $N^{1-(j-1)/K}$ we obtain an upper bound of 
\begin{equation*}
\begin{array}{l}
\sum_{j=1}^{K_{1}}N^{4-2(j-1)/K}\times N^{-2\frac{2K_{1}-j}{K}+5s\frac{%
2K_{1}-j}{K}}+N^{2}N^{-2+5s} \\ 
=N^{4}\sum_{j=1}^{K_{1}}N^{-4\left( 1-s^{3}\right) +2/K}N^{10s\left(
1-s^{3}\right) }+\sum_{j=1}^{K_{1}}N^{5s} \\ 
\leq \left[ K_{1}+1\right] N^{\frac{2}{K}+10s},%
\end{array}%
\end{equation*}%
which for $N$ large enough is at most $N^{\frac{2}{K}+11s}$.
\end{proof}

\subsection{Finite dimensional distribution of the harmonic averages\label%
{fdd}}

We now state and prove a result concerns the joint distribution of the
increment of the harmonic averages at mesoscopic scales. The next theorem
shows approximate joint Gaussianity of $\left\{ U_{m}\right\} _{m=1}^{K_{1}}$%
, defined in (\ref{Um}).

\begin{theorem}
\label{expcircle}For all bounded sequence $\left\{ \lambda _{m}\right\}
_{m=1,...,K_{1}}$ such that $\max_{m}\lambda _{m}\leq C$ and $v\in \left[
-0.9N,0.9N\right] ^{2}$, we have for all $K$ sufficiently large, 
\begin{equation}
\mathbb{E}^{D_{N},0}[\exp (\sum_{m=1}^{K_{1}}\lambda _{m}U_{m}(v))]=\exp (%
\frac{1}{2}\sum_{m=1}^{K_{1}}\lambda _{m}^{2}\frac{1}{K}g\log N+o\left( \log
N\right) ),  \label{1layer}
\end{equation}%
where the $o\left( \log N\right) $ term depends on $K,\varepsilon ,C$, and
the constant $\delta $ from Theorem \ref{decouple}. Also, for $%
v_{1},v_{2}\in \left[ -N/2,N/2\right] ^{2}$ such that for some $j\in \left\{
1,\ldots ,K_{1}\right\} $, $N^{1-\frac{j}{K}}\leq |v_{1}-v_{2}|\leq N^{1-%
\frac{j-1}{K}}$, and for bounded sequences $\left\{ \lambda _{m,i}\right\}
_{i=1,2}$ such that $\max_{m,i}\lambda _{m,i}\leq C$, we have for all $K$
sufficiently large,%
\begin{equation}
\begin{array}{l}
\mathbb{E}^{D_{N},0}[\exp (\sum_{m=1}^{K_{1}}\lambda
_{m,1}U_{m}(v_{1})+\sum_{m=j+1}^{K_{1}}\lambda _{m,2}U_{m}(v_{2})))] \\ 
=\exp (\frac{1}{2}\sum_{m=1}^{K_{1}}\lambda _{m,1}^{2}\frac{1}{K}g\log N+%
\frac{1}{2}\sum_{m=j+1}^{K_{1}}\lambda _{m,2}^{2}\frac{1}{K}g\log N+o\left(
\log N\right) ).%
\end{array}
\label{2layer}
\end{equation}
\end{theorem}

\begin{proof}
We first prove (\ref{1layer}). Recall that 
\begin{eqnarray*}
\mathcal{G} &=&\left\{ \phi :\max_{v\in D_{N}}\left\vert \phi \left(
v\right) \right\vert <\left( \log N\right) ^{2}\right\} \\
&=&\left\{ \phi :\max_{v\in D_{N}}\left\vert \phi \left( v\right)
\right\vert <c\left( s\right) \left( \log r_{M}\right) ^{2}\right\}
\end{eqnarray*}%
Using the Brascamp-Lieb inequality and Lemma \ref{bad}, it is easy to bound%
\begin{equation*}
\mathbb{E}^{D_{N},0}[\exp (\sum_{m=1}^{K_{1}}\lambda _{m}U_{m}(v))1_{%
\mathcal{G}^{c}}]=o_{N}\left( 1\right) ,
\end{equation*}%
therefore we only need to compute $\mathbb{E}^{D_{N},0}[\exp
(\sum_{m=1}^{K_{1}}\lambda _{m}U_{m}(v))1_{\mathcal{G}}]$.

Indeed, denote $r_{\left[ mM/K_{1}\right] }$ as $\tilde{r}_{m}$, and $%
\mathcal{F}_{m}=\sigma \left\{ \phi \left( v\right) :v\in D_{N}\setminus B_{%
\tilde{r}_{m}}\left( v\right) \right\} $, by the Markov property we have 
\begin{eqnarray*}
&&\mathbb{E}^{D_{N},0}[\exp (\sum_{m=1}^{K_{1}}\lambda _{m}U_{m}(v))1_{%
\mathcal{G}}] \\
&=&\mathbb{E}^{D_{N},0}\left[ \exp \left( \sum_{m=1}^{K_{1}-1}\lambda
_{m}U_{m}\right) 1_{\mathcal{G}}\mathbb{E}\left[ e^{\lambda
_{K_{1}}U_{K_{1}}}1_{\mathcal{G}}|\mathcal{F}_{K_{1}-1}\right] \right] .
\end{eqnarray*}%
By Lemma \ref{jasonerr}, there exist $C_{1}<\infty $ and $\delta >0$, such
that 
\begin{eqnarray*}
\left\vert \mathbb{E}\left[ e^{\lambda _{K_{1}}U_{K_{1}}}1_{\mathcal{G}}|%
\mathcal{F}_{K_{1}-1}\right] -\mathbb{E}^{\tilde{r}_{K_{1}-1},0}\left[
e^{\lambda _{K_{1}}U_{K_{1}}}1_{\mathcal{G}}\right] \right\vert &\leq &%
\tilde{r}_{K_{1}-1}^{-\delta }\exp \left( c_{1}\text{Var}_{DGFF}^{\tilde{r}%
_{K_{1}-1},0}\left( \lambda _{K_{1}}U_{K_{1}}\right) \right) \\
&\leq &\tilde{r}_{K_{1}-1}^{-\delta }\exp \left( C^{2}C_{1}\frac{1}{K}\log
N\right) ,
\end{eqnarray*}%
where $C=\max_{m}\lambda _{m}$. Take $K$ large enough such that 
\begin{equation*}
C^{2}C_{1}\frac{1}{K}\leq \frac{1}{2}\delta s^{3},
\end{equation*}%
we thus have 
\begin{equation*}
\left\vert \mathbb{E}\left[ e^{\lambda _{K_{1}}U_{K_{1}}}1_{\mathcal{G}}|%
\mathcal{F}_{K_{1}-1}\right] -\mathbb{E}^{\tilde{r}_{K_{1}-1},0}\left[
e^{\lambda _{K_{1}}U_{K_{1}}}1_{\mathcal{G}}\right] \right\vert \leq \tilde{r%
}_{K_{1}-1}^{-\delta /2}\leq \tilde{r}_{K_{1}-1}^{-\delta /2}\mathbb{E}^{%
\tilde{r}_{K_{1}-1},0}\left[ e^{\lambda _{K_{1}}U_{K_{1}}}1_{\mathcal{G}}%
\right] .
\end{equation*}

Therefore%
\begin{eqnarray*}
&&\mathbb{E}^{D_{N},0}[\exp (\sum_{m=1}^{K_{1}}\lambda _{m}U_{m}(v))1_{%
\mathcal{G}}] \\
&=&\left( 1+O\left( \tilde{r}_{K_{1}-1}^{-\delta /2}\right) \right) \mathbb{E%
}^{\tilde{r}_{K_{1}-1},0}\left[ e^{\lambda _{K_{1}}U_{K_{1}}}1_{\mathcal{G}}%
\right] \mathbb{E}^{D_{N},0}\left[ \exp \left( \sum_{m=1}^{K_{1}-1}\lambda
_{m}U_{m}\right) 1_{\mathcal{G}}\right] .
\end{eqnarray*}%
Keep iterating then yields%
\begin{equation*}
\mathbb{E}^{D_{N},0}[\exp (\sum_{m=1}^{K_{1}}\lambda _{m}U_{m}(v))1_{%
\mathcal{G}}]=\prod_{m=1}^{K_{1}}\left( 1+O\left( \tilde{r}_{m-1}^{-\delta
/2}\right) \right) \mathbb{E}^{\tilde{r}_{m-1},0}\left[ e^{\lambda
_{m}U_{m}}1_{\mathcal{G}}\right] .
\end{equation*}

By Theorem \ref{Gauss} (and Remark \ref{Gauss2}), there exists $g=g\left(
V\right) >0$, such that 
\begin{equation}
\mathbb{E}^{\tilde{r}_{m-1},0}\left[ e^{\lambda _{m}U_{m}}\right] =\exp
\left( \frac{\lambda _{m}^{2}}{2}\frac{g}{K}\log N+o\left( \log N\right)
\right) ,  \label{mgfu}
\end{equation}%
and by Lemma \ref{bad} and the Brascamp-Lieb inequality,%
\begin{equation*}
\mathbb{E}^{\tilde{r}_{m-1},0}\left[ e^{\lambda _{m}U_{m}}1_{\mathcal{G}^{c}}%
\right] =o_{N}\left( 1\right) \text{.}
\end{equation*}%
Since $\sum_{m=1}^{K_{1}}\tilde{r}_{m-1}^{-\delta /2}<\infty $, this
finishes the proof of (\ref{1layer}).

The proof of (\ref{2layer}) is very similar to that of (\ref{1layer}). We
define for $i=1,2$, $\mathcal{F}_{m,i}=\sigma \left\{ \phi \left( v\right)
:v\in D_{N}\setminus B_{\tilde{r}_{m}}\left( v_{i}\right) \right\} $. Then
by the same argument, 
\begin{eqnarray*}
&&\mathbb{E}^{D_{N},0}[\exp (\sum_{m=1}^{K_{1}}\lambda
_{m,1}U_{m}(v_{1})+\sum_{m=j+1}^{K_{1}}\lambda _{m,2}U_{m}(v_{2}))1_{%
\mathcal{G}}] \\
&=&\mathbb{E}^{D_{N},0}[\exp (\sum_{m=1}^{K_{1}-1}\lambda
_{m,1}U_{m}(v_{1})+\sum_{m=j+1}^{K_{1}}\lambda _{m,2}U_{m}(v_{2}))1_{%
\mathcal{G}}\mathbb{E}\left[ \exp \left( \lambda
_{K_{1},1}U_{K_{1}}(v_{1})\right) 1_{\mathcal{G}}|\mathcal{F}_{K_{1}-1,1}%
\right] ] \\
&=&\left( 1+O\left( \tilde{r}_{K_{1}-1}^{-\delta /2}\right) \right) \mathbb{E%
}^{\tilde{r}_{K_{1}-1},0}\left[ \exp \left( \lambda
_{K_{1},1}U_{K_{1}}(v_{1})\right) 1_{\mathcal{G}}\right] \\
&&\mathbb{E}^{D_{N},0}[\exp (\sum_{m=1}^{K_{1}-1}\lambda
_{m,1}U_{m}(v_{1})+\sum_{m=j+1}^{K_{1}}\lambda _{m,2}U_{m}(v_{2}))1_{%
\mathcal{G}}].
\end{eqnarray*}%
Then conditioned on $\mathcal{F}_{K_{1}-1,2}$, apply the Markov property and
Lemma \ref{jasonerr}, we can write the above display as 
\begin{eqnarray*}
&&\left( 1+O\left( \tilde{r}_{K_{1}-1}^{-\delta /2}\right) \right) \mathbb{E}%
^{\tilde{r}_{K_{1}-1},0}\left[ \exp \left( \lambda
_{K_{1},1}U_{K_{1}}(v_{1})\right) 1_{\mathcal{G}}\right] \mathbb{E}^{\tilde{r%
}_{K_{1}-1},0}\left[ \exp \left( \lambda _{K_{1},2}U_{K_{1}}(v_{2})\right)
1_{\mathcal{G}}\right] \\
&&\times \mathbb{E}^{D_{N},0}[\exp (\sum_{m=1}^{K_{1}-1}\lambda
_{m,1}U_{m}(v_{1})+\sum_{m=j+1}^{K_{1}-1}\lambda _{m,2}U_{m}(v_{2}))1_{%
\mathcal{G}}].
\end{eqnarray*}%
Keep iterating, we obtain 
\begin{eqnarray*}
&&\mathbb{E}^{D_{N},0}[\exp (\sum_{m=1}^{K_{1}}\lambda
_{m,1}U_{m}(v_{1})+\sum_{m=j+1}^{K_{1}}\lambda _{m,2}U_{m}(v_{2}))1_{%
\mathcal{G}}] \\
&=&\prod_{m=j+1}^{K_{1}}\left( 1+O\left( \tilde{r}_{m-1}^{-\delta /2}\right)
\right) \mathbb{E}^{\tilde{r}_{m-1},0}\left[ \exp \left( \lambda
_{m,1}U_{m}(v_{1})\right) 1_{\mathcal{G}}\right] \mathbb{E}^{\tilde{r}%
_{m-1},0}\left[ \exp \left( \lambda _{m,2}U_{m}(v_{2})\right) 1_{\mathcal{G}}%
\right] \\
&&\times \prod_{m=1}^{j}\left( 1+O\left( \tilde{r}_{m-1}^{-\delta /2}\right)
\right) \mathbb{E}^{\tilde{r}_{m-1},0}\left[ \exp \left( \lambda
_{m,1}U_{m}(v_{1})\right) 1_{\mathcal{G}}\right] .
\end{eqnarray*}%
Applying (\ref{mgfu}) we conclude the proof of (\ref{2layer}).
\end{proof}

\subsection{Bootstrapping\label{boot}}

We now use Proposition \ref{prop: weak bound} to prove the desired lower
bound (\ref{lb}). Proposition \ref{prop: weak bound} shows that the field
reaches $\left( 1-2s \right) 2\sqrt{g}\log N$ with at least polynomially
small probability. We will apply Theorem \ref{decouple} to see that the
field in different regions of $\left[ -N,N\right] ^{2}$ are essentially
decoupled. Therefore applying Proposition \ref{prop: weak bound} in each
region one can show with high probability, there is some $v\in \left[ -N,N%
\right] ^{2}$ such that $\phi (v)-X_{r_{_{0},-}}(v)\geq (1-2s )2\sqrt{g}\log
N$.

To carry out this argument, tile $\left[ -N,N\right] ^{2}$ by disjoint boxes 
$D_{1},D_{2},\ldots ,D_{m}$ of side-length $N^{1-\eta }$, where $m\asymp
N^{\eta }$, and $\eta $ is a small number that will be chosen later. Let $%
\mathcal{B}$ be the union of all the $\partial D_{i}$.

Consider the good event 
\begin{equation*}
\mathcal{G}=\{\max_{v\in \left[ -N,N\right] ^{2}}|\phi (v)|\leq \left( \log
N\right) ^{2}\}.
\end{equation*}%
By Lemma \ref{bad}, we have $\mathbb{P}^{D_{N},0}[\mathcal{G}%
^{c}]\preccurlyeq e^{-c\left( \log N\right) ^{3}}$, as $N\rightarrow \infty $%
.

On the event $\mathcal{G}$, for $i=1,...,m$, let $\bar{D}_{i}$ be the box
concentric to $D_{i}$, but with side length $\frac{1}{2}N^{1-\eta }$. Let $R=%
\frac{1}{2}N^{1-\eta }$. We further define 
\begin{equation*}
\tilde{\mathcal{N}_{K}}_{i}=\left\{ \forall v\in \bar{D}_{i}:\phi \left(
v\right) -X_{R,-}(v,\phi )<(1-2s)\left( 1-\eta \right) 2\sqrt{g}\log
N\right\} .
\end{equation*}%
Now 
\begin{equation*}
\mathbb{P}^{D_{N},0}[\tilde{\mathcal{N}_{K}}_{i},i=1,...,m;\mathcal{G}]=%
\mathbb{P}^{D_{N},0}[\mathbb{P}[\tilde{\mathcal{N}_{K}}_{i},i=1,...,m|\phi
(x),x\in \mathcal{B}];\mathcal{G}].
\end{equation*}%
Using the Gibbs property of the measure (\ref{GL}), we have the conditional
decoupling 
\begin{equation*}
\mathbb{P}^{D_{N},0}[\tilde{\mathcal{N}_{K}}_{i},i=1,...,m|\phi (x),x\in 
\mathcal{B}]=\prod_{i=1}^{m}\mathbb{P}^{D_{i},\phi 1_{\partial D_{i}}}[%
\tilde{\mathcal{N}_{K}}_{i}|\phi (x),x\in \partial D_{i}].  \label{eq: bla}
\end{equation*}%
Consider for each $i$ the law $\mathbb{P}^{D_{i},\phi 1_{\partial D_{i}}}$.
Then on $\mathcal{G}$ we can apply Lemma \ref{decouple} and \ref{average}$\ $%
to construct a coupling $Q^{i}$ of a field $\phi $ with law $\mathbb{P}%
^{D_{i},\phi 1_{\partial D_{i}}}$ and a field $\phi ^{0,i}$ with law $%
\mathbb{P}^{D_{i},0}$ such that 
\begin{equation*}
Q^{i}[\forall v\in \bar{D}_{i}:\phi \left( v\right) -X_{R,-}(v,\phi )=\phi
^{0,i}\left( v\right) -X_{R,-}^{{}}(v,\phi ^{0,i})]\geq 1-N^{-\delta \left(
1-\eta \right) },
\end{equation*}%
where the constant $\delta >0$ is from Theorem \ref{decouple}.

Thus 
\begin{equation*}
\begin{array}{l}
\mathbb{P}^{D_{N},0}\left( 
\begin{array}{c}
\forall v\in \left[ -0.9N,0.9N\right] ^{2}: \\ 
\phi \left( v\right) -X_{R,-}(v,\phi )<(1-2s)\left( 1-\eta \right) 2\sqrt{g}%
\log N;\mathcal{G}%
\end{array}%
\right)  \\ 
\leq \prod_{i=1}^{m}\left( \mathbb{P}^{D_{i},0}[\tilde{\mathcal{N}_{K}}%
_{i}]+N^{-\delta \left( 1-\eta \right) }\right)  \\ 
\leq \prod_{i=1}^{m}\left( 1-\left( N^{1-\eta }\right) ^{-21s}+N^{-\delta
\left( 1-\eta \right) }\right) ,%
\end{array}%
\end{equation*}%
where we apply Proposition \ref{prop: weak bound} to obtain the last
inequality. Now let $s$ and $\eta $ be small enough, depending on $\delta $,
such that 
\begin{equation}
21s<\delta \text{ and }\eta >21s/\left( 1+21s\right) .  \label{etaeps}
\end{equation}%
Thus we have 
\begin{equation}
\mathbb{P}^{D_{N},0}\left( 
\begin{array}{c}
\forall v\in \left[ -0.9N,0.9N\right] ^{2}: \\ 
\phi \left( v\right) -X_{R,-}(v,\phi )<(1-2s)\left( 1-\eta \right) 2\sqrt{g}%
\log N;\mathcal{G}%
\end{array}%
\right) \preccurlyeq e^{-N^{\varepsilon _{1}}},  \label{scale1}
\end{equation}%
for some $\varepsilon _{1}>0$.

In view of (\ref{etaeps}), we can take $\eta =21s$. Then, on the complement
of the event (\ref{scale1}), there exists $v_{1}\in \left[ -0.9N,0.9N\right]
^{2}$ such that 
\begin{equation*}
\phi \left( v_{1}\right) -X_{R,-}(v_{1},\phi )\geq (1-19s)2\sqrt{g}\log N.
\end{equation*}%
Notice that (for $g_{0}=2/\pi $) 
\begin{equation*}
\text{Var}_{DGFF}^{D_{N},0}\left[ X_{R,-}(v_{1},\phi )\right] =g_{0}\eta
\log N+o\left( \log N\right) =21sg_{0}\log N+o\left( \log N\right) .
\end{equation*}%
By Lemma \ref{lem: BL tail bound}, there exists $c_{BL}>0$, such that%
\begin{equation}
\mathbb{P}^{D_{N},0}\left[ X_{R,-}(v_{1},\phi )>s^{1/3}\log N\right] \leq
\exp \left( -c_{BL}\frac{s^{2/3}\left( \log N\right) ^{2}}{s\log N}\right)
=N^{-c_{BL}s^{-1/3}}.  \label{first}
\end{equation}%
Combining (\ref{scale1}) and (\ref{first}), we see that%
\begin{equation*}
\mathbb{P}^{D_{N},0}\left[ \max_{v\in \left[ -0.9N,0.9N\right] ^{2}}\phi
\left( v\right) <(1-2s^{1/3})2\sqrt{g}\log N\right] \leq
N^{-c_{BL}s^{-1/3}}+e^{-N^{\varepsilon _{1}}}.
\end{equation*}%
And we conclude (\ref{lb}).

\subsection{High points\label{hp}}

We now sketch the proof of Theorem \ref{thm: high points}. The proof follows
from the same argument as the proof of Theorem \ref{main}, for completeness
we sketch the idea below.

It suffices to prove that for any $s>0$, 
\begin{eqnarray}
\mathbb{P}^{D_{N},0}\left( \left\vert \mathcal{H}_{N}\left( \eta \right)
\right\vert >N^{2\left( 1-\eta ^{2}\right) +s}\right)  &=&o_{N}\left(
1\right) ,\text{ and}  \label{upp} \\
\mathbb{P}^{D_{N},0}\left( \left\vert \mathcal{H}_{N}\left( \eta \right)
\right\vert <N^{2\left( 1-\eta ^{2}\right) -s}\right)  &=&o_{N}\left(
1\right) .  \label{low}
\end{eqnarray}%
Since 
\begin{eqnarray*}
\mathbb{P}^{D_{N},0}\left( \left\vert \mathcal{H}_{N}\left( \eta \right)
\right\vert >N^{2\left( 1-\eta ^{2}\right) +s}\right)  &\leq &N^{-2\left(
1-\eta ^{2}\right) -s}\mathbb{E}\left[ \left\vert \mathcal{H}_{N}\left( \eta
\right) \right\vert \right]  \\
&\leq &N^{-2\left( 1-\eta ^{2}\right) -s}\sum_{v\in D_{N}}\mathbb{P}%
^{D_{N},0}\left( \phi \left( v\right) \geq 2\sqrt{g}\eta \log N\right) ,
\end{eqnarray*}%
the upper bound (\ref{upp}) follows directly from applying Theorem \ref{thm:
tail bound} with $u=2\sqrt{g}\eta \log N$.

We now focus on the lower bound (\ref{low}). Recall the definition of $U_{m}$
in (\ref{Um}). For $\eta \in \left( 0,1\right) $, having in mind that we aim
to count the points $\left\{ v\in D_{N}:\phi \left( v\right) >2\sqrt{g}\eta
\log N\right\} $, we look at the following truncated event such that the
increments $U_{m}$ are slightly higher than $2\sqrt{g}\frac{\eta }{K}\log N$%
: 
\begin{equation*}
J_{m}(v;\eta ;s)=\left\{ U_{m}\left( v\right) \in \left[ (1+s)2\sqrt{g}\frac{%
\eta }{K}\log N,(1+2s)2\sqrt{g}\frac{\eta }{K}\log N\right] \right\} .
\end{equation*}%
and for $K_{1}:=\left[ \left( 1-s^{3}\right) K\right] +1$, 
\begin{equation*}
J(v;\eta ;s)=\bigcap_{m=1,\ldots ,K_{1}}J_{m}(v;\eta ;s).
\end{equation*}%
Also define the counting random variable 
\begin{equation*}
\mathcal{N}_{K_{1}}(\eta ,s)=\sum_{v\in \left[ -0.9N,0.9N\right]
^{2}}1_{J(v;\eta ;s)}.
\end{equation*}%
By the same Brascamp-Lieb bounds as (\ref{last}) and (\ref{xerr}), to study
the dimension of $\mathcal{H}_{N}\left( \eta \right) $, it suffices to study 
$\left\{ v:J(v;\eta ;s)\text{ occurs}\right\} $. Indeed, the same first
moment computation as Lemma \ref{lem: fm} and Lemma \ref{1pt} (but instead
using the change of measure $\frac{dQ}{d\mathbb{P}^{D_{N},0}}=\frac{\exp
(\lambda \eta \sum_{m=1}^{K_{1}}U_{m}(v))}{\mathbb{E}^{D_{N},0}[\exp
(\lambda \eta \sum_{m=1}^{K_{1}}U_{m}(v))]}$) yields%
\begin{equation*}
\mathbb{E}\left[ \mathcal{N}_{K_{1}}(\eta ,s)\right] \geq N^{2\left( 1-\eta
^{2}\right) -8s\eta ^{2}},
\end{equation*}%
and the same second moment computation as Lemma \ref{lem: sm} and Lemma \ref%
{2pt} yields 
\begin{equation*}
\mathbb{E}\left[ \mathcal{N}_{K_{1}}^{2}(\eta ,s)\right] \leq N^{4\left(
1-\eta ^{2}\right) -5s\eta ^{2}}.
\end{equation*}%
Therefore 
\begin{equation*}
\mathbb{E}\left[ \mathcal{N}_{K_{1}}^{2}(\eta ,s)\right] \leq N^{11s\eta
^{2}}\mathbb{E}\left[ \mathcal{N}_{K_{1}}(\eta ,s)\right] ^{2}.
\end{equation*}

Applying the Payley-Zygmund inequality then yields%
\begin{eqnarray*}
&&\mathbb{P}^{D_{N},0}\left( \left\vert \left\{ v:J(v;\eta ;s)\text{ occurs}%
\right\} \right\vert <\frac{1}{2}N^{2\left( 1-\eta ^{2}\right) -s}\right)  \\
&\leq &1-\mathbb{P}^{D_{N},0}\left( \mathcal{N}_{K_{1}}(\eta ,s)>\frac{1}{2}%
\mathbb{E}\left[ \mathcal{N}_{K_{1}}(\eta ,s)\right] \right)  \\
&\leq &1-cN^{-11s\eta ^{2}}.
\end{eqnarray*}%
But to complete the proof of (\ref{low}) we want $\mathbb{P}^{D_{N},0}\left( 
\mathcal{N}_{K_{1}}(\eta ,s)>\frac{1}{2}\mathbb{E}\left[ \mathcal{N}%
_{K_{1}}(\eta ,s)\right] \right) $ to be close to $1$. This can be proved by
carrying out the same bootstrapping in Section \ref{boot}, obtaining the
high probability by creating a large number ($N^{\gamma }$, where $\gamma
=\gamma \left( s,\delta \right) $, and $\delta $ is the constant from
Theorem \ref{decouple}) of essentially independent trials with success
probability $N^{-11s\eta ^{2}}$.

\bigskip

\paragraph{\textbf{Acknowledgments:}}

We thank Ron Peled for very helpful discussions and helping us generalize
the results in an earlier version of this paper; Ofer Zeitouni for helpful
discussions; Thomas Spencer for discussions on gradient field models and the
arguments in \cite{CS}; and Jason Miller for useful commuications. The work
of W.W. is supported in part by NSF grant DMS-1507019. The work was started
while both authors are at Courant Institute, NYU. Part of the work is done
when the last author was visiting Weizmann Institute, Tel-Aviv University
and NYU Shanghai, and we thank these institutes for their hospitality.

\bibliographystyle{plain}
\bibliography{GLmax}

\begin{thebibliography}{ABBS13}

\bibitem[ABB17]{ArgBelBou15}
L.-P. Arguin, D.~Belius, and P.~Bourgade.
\newblock {Maximum of the characteristic polynomial of random unitary
  matrices}.
\newblock {\em Comm. Math. Phys.}, 349(2):703--751, 2017.

\bibitem[ABBS13]{AidekonBerestickyBruneyShiBBMSeenFromItsTip}
E.~A{\"{\i}}d{\'e}kon, J.~Berestycki, {\'E}.~Brunet, and Z.~Shi.
\newblock {Branching {B}rownian motion seen from its tip}.
\newblock {\em Probab. Theory Related Fields}, 157(1-2):405--451, 2013.

\bibitem[ABH17]{ArguinBeliusHarper-RandomZeta}
L.-P. Arguin, D.~Belius, and A.~Harper.
\newblock {Maxima of a randomized {R}iemann zeta function, and branching random
  walks}.
\newblock {\em Ann. Appl. Probab.}, 27(1):178--215, 2017.

\bibitem[ABK13]{ABKExtremalProcOFBBM}
L.-P. Arguin, A.~Bovier, and N.~Kistler.
\newblock {The extremal process of branching {B}rownian motion}.
\newblock {\em Probab. Theory Related Fields}, 157(3-4):535--574, 2013.

\bibitem[A{\"i}d13]{AidekonConvinLawofMinofBRW}
E.~A{\"i}d{\'e}kon.
\newblock {Convergence in law of the minimum of a branching random walk}.
\newblock {\em Ann. Probab.}, 41(3A):1362--1426, 2013.

\bibitem[AW]{AW}
Scott Armstrong and Wei Wu.
\newblock {\em in preparation}.

\bibitem[BDG01]{BDG}
Erwin Bolthausen, Jean-Dominique Deuschel, and Giambattista Giacomin.
\newblock Entropic repulsion and the maximum of the two-dimensional harmonic
  crystal.
\newblock {\em The Annals of Probability}, 29(4):1670--1692, 2001.

\bibitem[BDZ16a]{BramsonDingZeitouni-ConvergenceinLawOfTheMaxOfNonLAtticeBRW}
Maury Bramson, Jian Ding, and Ofer Zeitouni.
\newblock Convergence in law of the maximum of nonlattice branching random
  walk.
\newblock {\em Ann. Inst. H. Poincar\'e Probab. Statist.}, 52(4):1897--1924, 11
  2016.

\bibitem[BDZ16b]{BDZ2}
Maury Bramson, Jian Ding, and Ofer Zeitouni.
\newblock Convergence in law of the maximum of the two-dimensional discrete
  {G}aussian free field.
\newblock {\em Communications on Pure and Applied Mathematics}, 69(1):62--123,
  2016.

\bibitem[BK16]{BeliusKister2DCT}
D.~Belius and N.~Kistler.
\newblock {The subleading order of two dimensional cover times}.
\newblock {\em Probability Theory and Related Fields}, pages 1--92, 2016.

\bibitem[BL76]{BL}
Herm~Jan Brascamp and Elliott~H Lieb.
\newblock On extensions of the {B}runn-{M}inkowski and {P}r{\'e}kopa-{L}eindler
  theorems, including inequalities for log concave functions, and with an
  application to the diffusion equation.
\newblock {\em Journal of Functional Analysis}, 22(4):366--389, 1976.

\bibitem[BL13]{BL2}
Marek Biskup and Oren Louidor.
\newblock Extreme local extrema of two-dimensional discrete {G}aussian free
  field.
\newblock {\em arXiv preprint arXiv:1306.2602}, 2013.

\bibitem[BL14]{BiskupLouidorConformalSymmetries}
M.~{Biskup} and O.~{Louidor}.
\newblock {Conformal symmetries in the extremal process of two-dimensional
  discrete Gaussian Free Field}.
\newblock {\em ArXiv e-prints}, October 2014.

\bibitem[BL16]{BiskupLouidorFullExtProcClusterLawEtc}
M.~{Biskup} and O.~{Louidor}.
\newblock {Full extremal process, cluster law and freezing for two-dimensional
  discrete Gaussian Free Field}.
\newblock {\em ArXiv e-prints}, June 2016.

\bibitem[BLL75]{BLL}
Herm~Jan Brascamp, Elliot~H Lieb, and Joel~L Lebowitz.
\newblock The statistical mechanics of anharmonic lattices.
\newblock In {\em Statistical Mechanics}, pages 379--390. Springer, 1975.

\bibitem[Bra78]{BramsonMaxDisplacementofBBM}
M.~Bramson.
\newblock {Maximal displacement of branching Brownian motion}.
\newblock {\em Comm. Pure Appl. Math}, 31(5):531--581, 1978.

\bibitem[Bra83]{Bramson1ConvergenceofSolutionsOfKolmogorovEqn}
M.~Bramson.
\newblock {Convergence of solutions of the Kolmogorov equation to traveling
  waves}.
\newblock {\em Memoirs of the American Mathematical Society}, 44(285):1--190,
  1983.

\bibitem[BY90]{BY}
David Brydges and Horng-Tzer Yau.
\newblock Grad $\phi$ perturbations of massless {G}aussian fields.
\newblock {\em Communications in Mathematical Physics}, 129(2):351--392, 1990.

\bibitem[BZ12]{BZ}
Maury Bramson and Ofer Zeitouni.
\newblock Tightness of the recentered maximum of the two-dimensional discrete
  {G}aussian free field.
\newblock {\em Communications on Pure and Applied Mathematics}, 65(1):1--20,
  2012.

\bibitem[CMN16]{ChaMadNaj16}
R.~Chhaibi, T.~Madaule, and J.~Najnudel.
\newblock {On the maximum of the C$\beta$E field}.
\newblock {\em Preprint arXiv:1607.00243}, 2016.

\bibitem[CS14]{CS}
Joseph~G Conlon and Thomas Spencer.
\newblock A strong central limit theorem for a class of random surfaces.
\newblock {\em Communications in Mathematical Physics}, 325(1):1--15, 2014.

\bibitem[Dav06]{Dav}
Olivier Daviaud.
\newblock Extremes of the discrete two-dimensional {G}aussian free field.
\newblock {\em The Annals of Probability}, 34(3):962--986, 2006.

\bibitem[DG00]{DG}
Jean-Dominique Deuschel and Giambattista Giacomin.
\newblock Entropic repulsion for massless fields.
\newblock {\em Stochastic processes and their applications}, 89(2):333--354,
  2000.

\bibitem[DGI00]{DGI}
Jean-Dominique Deuschel, Giambattista Giacomin, and Dmitry Ioffe.
\newblock Large deviations and concentration properties for $\nabla \phi$
  interface models.
\newblock {\em Probability Theory and Related Fields}, 117(1):49--111, 2000.

\bibitem[DRZ15]{DRZ}
Jian Ding, Rishideep Roy, and Ofer Zeitouni.
\newblock Convergence of the centered maximum of log-correlated {G}aussian
  fields.
\newblock {\em arXiv preprint arXiv:1503.04588}, 2015.

\bibitem[DS11]{DuplantierSheffieldLQGandKPZ}
B.~Duplantier and S.~Sheffield.
\newblock Liouville quantum gravity and {KPZ}.
\newblock {\em Invent. Math.}, 185(2):333--393, 2011.

\bibitem[DZ14]{DZ}
Jian Ding and Ofer Zeitouni.
\newblock Extreme values for two-dimensional discrete {G}aussian free field.
\newblock {\em The Annals of Probability}, 42(4):1480--1515, 2014.

\bibitem[FS97]{FS}
T~Funaki and Herbert Spohn.
\newblock Motion by mean curvature from the {G}inzburg-{L}andau interface
  model.
\newblock {\em Communications in Mathematical Physics}, 185(1):1--36, 1997.

\bibitem[GK80]{GK}
K~Gawedzki and A~Kupiainen.
\newblock A rigorous block spin approach to massless lattice theories.
\newblock {\em Communications in Mathematical Physics}, 77(1):31--64, 1980.

\bibitem[GOS01]{GOS}
Giambattista Giacomin, Stefano Olla, and Herbert Spohn.
\newblock Equilibrium fluctuations for $\nabla \phi$ interface model.
\newblock {\em Annals of Probability}, pages 1138--1172, 2001.

\bibitem[Hel02]{Hel}
Bernard Helffer.
\newblock {\em Semiclassical analysis, Witten Laplacians, and statistical
  mechanics}, volume~1.
\newblock World Scientific, 2002.

\bibitem[HMP10]{HuMillerPeresThickPointsofGFF}
Xiaoyu Hu, Jason Miller, and Yuval Peres.
\newblock {Thick points of the {G}aussian free field}.
\newblock {\em Ann. Probab.}, 38(2):896--926, 2010.

\bibitem[HS94]{HS}
Bernard Helffer and Johannes Sj{\"o}strand.
\newblock On the correlation for {K}ac-like models in the convex case.
\newblock {\em Journal of Statistical Physics}, 74(1-2):349--409, 1994.

\bibitem[Ken00]{K1}
Richard Kenyon.
\newblock Conformal invariance of domino tiling.
\newblock {\em Annals of Probability}, pages 759--795, 2000.

\bibitem[Ken01]{K2}
Richard Kenyon.
\newblock Dominos and the {G}aussian free field.
\newblock {\em Annals of Probability}, pages 1128--1137, 2001.

\bibitem[Kis15]{KistlerDerridasRandomEnergyModelsBeyondSpinGlasses}
N.~Kistler.
\newblock {\em Derrida's Random Energy Models}, pages 71--120.
\newblock Springer International Publishing, Cham, 2015.

\bibitem[Law08]{Law}
Gregory~F Lawler.
\newblock {\em Conformally invariant processes in the plane}.
\newblock Number 114. American Mathematical Soc., 2008.

\bibitem[LL10]{LL}
Gregory~F Lawler and Vlada Limic.
\newblock {\em Random walk: a modern introduction}, volume 123.
\newblock Cambridge University Press, 2010.

\bibitem[LP16]{LamPaq16}
G.~{Lambert} and E.~{Paquette}.
\newblock {The law of large numbers for the maximum of almost Gaussian
  log-correlated fields coming from random matrices}.
\newblock {\em ArXiv e-prints}, November 2016.

\bibitem[Mad17]{MadauleConvergenceInLawOftheBRWSeenFromItsTip}
T.~Madaule.
\newblock Convergence in law for the branching random walk seen from its tip.
\newblock {\em J. Theoret. Probab.}, 30(1):27--63, 2017.

\bibitem[Mil10]{M2}
Jason Miller.
\newblock Universality for {SLE} (4).
\newblock {\em arXiv preprint arXiv:1010.1356}, 2010.

\bibitem[Mil11]{M}
Jason Miller.
\newblock Fluctuations for the {G}inzburg-{L}andau $\nabla \phi$ interface
  model on a bounded domain.
\newblock {\em Communications in Mathematical Physics}, 308(3):591--639, 2011.

\bibitem[NS97]{NS}
Ali Naddaf and Thomas Spencer.
\newblock On homogenization and scaling limit of some gradient perturbations of
  a massless free field.
\newblock {\em Communications in Mathematical Physics}, 183(1):55--84, 1997.

\bibitem[PZ17]{PaqZei16}
E.~Paquette and O.~Zeitouni.
\newblock The maximum of the cue field.
\newblock {\em International Mathematics Research Notices}, page rnx033, 2017.

\bibitem[RV08]{robert2008hydrodynamic}
Raoul Robert and Vincent Vargas.
\newblock Hydrodynamic turbulence and intermittent random fields.
\newblock {\em Communications in Mathematical Physics}, 284(3):649--673, 2008.

\bibitem[She07]{SheffieldGFFforMath}
S.~Sheffield.
\newblock {Gaussian free fields for mathematicians}.
\newblock {\em Probab. Theory Related Fields}, 139(3-4):521--541, 2007.

\bibitem[SW12]{SW}
Scott Sheffield and Wendelin Werner.
\newblock Conformal loop ensembles: the {M}arkovian characterization and the
  loop-soup construction.
\newblock {\em Annals of Mathematics}, 176(3):1827--1917, 2012.

\end{thebibliography}

\end{document}